\documentclass{mcom-l}
\usepackage[T1]{fontenc}
\usepackage{amssymb}
\usepackage{amsmath}
\usepackage{amscd}
\usepackage{amsfonts}
\usepackage{textcomp}
\usepackage{amssymb}
\usepackage{latexsym}
\usepackage{url}
\usepackage{epsfig}
\usepackage{graphics}
\usepackage{graphicx}
\usepackage{algorithmic}
\usepackage{algorithm}
\usepackage{multirow}
\usepackage{pgf}
\DeclareFontFamily{OT1}{pzc}{}
\DeclareFontShape{OT1}{pzc}{m}{it}{<-> s * [1.10] pzcmi7t}{}
\DeclareMathAlphabet{\mathpzc}{OT1}{pzc}{m}{it}
\newcommand{\Elld}{\mathpzc{Ell}}
\newcommand{\F}{\mathbb{F}}

\newcommand{\Z}{\mathbb{Z}}

\newcommand{\C}{\mathbb{C}}
\newcommand{\PP}{\mathcal{P}}

\newtheorem{theorem}{Theorem}[section]
\newtheorem{lemma}[theorem]{Lemma}
\newtheorem{proposition}[theorem]{Proposition}
\newtheorem{corollary}[theorem]{Corollary}

\theoremstyle{definition}
\newtheorem{definition}[theorem]{Definition}
\newtheorem{example}[theorem]{Example}

\theoremstyle{remark}
\newtheorem{remark}[theorem]{Remark}

\numberwithin{equation}{section}

\begin{document}

\title{Pairing the volcano}
\author{Sorina Ionica}
\address{Laboratoire d'Informatique de l'Ecole Polytechnique (LIX)\footnote{This work has been carried out at Prism Laboratory, University of Versailles and is part of the author's Phd thesis.},
91128 Palaiseau CEDEX, France}
\email{sorina.ionica@m4x.org}

\author{Antoine Joux}
\address{DGA and Universit\'e de Versailles Saint-Quentin-en-Yvelines,
45 avenue des \'Etats-Unis\\
78035 Versailles CEDEX, France}
\email{antoine.joux@m4x.org}


\begin{abstract}
  Isogeny volcanoes are graphs whose vertices are elliptic curves and
  whose edges are $\ell$-isogenies. Algorithms allowing to travel on
  these graphs were developed by Kohel in his thesis (1996) and later
  on, by Fouquet and Morain (2001). However, up to now, no method was
  known, to predict, before taking a step on the volcano, the
  direction of this step. Hence, in Kohel's and Fouquet-Morain
  algorithms, many steps are taken before choosing the right direction.
  In particular, ascending or horizontal isogenies are usually found
  using a trial-and-error approach. In this paper, we propose an
  alternative method that efficiently finds all points $P$ of order
  $\ell$ such that the subgroup generated by $P$ is the kernel of an
  horizontal or an ascending isogeny. In many cases, our method is
  faster than previous methods. This is an extended version of a paper published in the proceedings of ANTS 2010. In addition, we treat the case of 2-isogeny volcanoes and we derive from the group structure of the curve and the pairing a new invariant of the endomorphism class of an elliptic curve. Our benchmarks show that the resulting algorithm for endomorphism ring computation is faster than Kohel's method for computing the $\ell$-adic valuation of the conductor of the endomorphism ring for small $\ell$.
\end{abstract}

\maketitle{}
\section{Introduction}

Let $E$ be an elliptic curve defined over a finite field $\F_q$, where $q=p^r$ is a prime power.
Let $\pi$ be the Frobenius endomorphism, i.e., $\pi(x,y)\mapsto(x^q,y^q)$ and denote by $t$ its trace.
Assume that $E$ is an ordinary curve and let $\mathcal{O}_E$ denotes
its ring of endomorphisms. We know~\cite[Th.~V.3.1]{Siv} that $\mathcal{O}_E$
is an order in an imaginary quadratic field $K$.
Let $d_{\pi} = t^2-4q$ be the discriminant of $\pi$. We can write $d_{\pi}=g^2d_K$, where $d_K$ is the discriminant
of the quadratic field $K$. There are only a finite number of possibilities for $\mathcal{O}_E$, since
$\Z[\pi]\subset \mathcal{O}_E\subset \mathcal{O}_{d_K}.$
Indeed, this requires that $f$, the conductor of $\mathcal{O}_E$, divides
$g$, the conductor of $\Z[\pi]$.
The cardinality of $E$ over $\F_q$ is $\#E(\F_q)=q+1-t$.
Two isogenous elliptic curves over $\F_q$ have the same cardinality, and thus the same trace $t$. In his thesis~\cite{Kohel},
Kohel studies how curves in $\textrm{Ell}_t(\F_q)$, the set of curves defined over $\F_q$ with
trace $t$, are related via isogenies of degree $\ell$. More
precisely, he describes the structure of the graph of $\ell$-isogenies
defined on $\textrm{Ell}_t(\F_q)$. He relates this
graph to orders in $\mathcal{O}_K$ and uses modular polynomials to
find the conductor of $\textrm{End}(E)$.


Fouquet and Morain~\cite{FouMor} call the connected components of this
graph \textit{isogeny volcanoes} and show that it is possible to travel through
these structures using modular polynomials, even without knowing the cardinality of the curve. Moreover, they compute the $\ell$-adic
valuation of the trace $t$, for $\ell|g$ and hence obtain some information on the cardinality of the curve.
 Recently, more applications of isogeny volcanoes were found: the computation of Hilbert class polynomials~\cite{BelBro08,Sutherland1},
of modular polynomials~\cite{Sutherland3} and of endomorphism rings of elliptic curves~\cite{Sutherland2}.

All the above methods make use of algorithms for traveling efficiently
on volcanoes. These algorithms need to walk on the crater, to
descend from the crater to the floor or to ascend from the floor to
the crater. In many cases, the structure of the $\ell$-Sylow subgroup
of the elliptic curve, allows one, after taking a step on the volcano, to
decide whether this step is ascending, descending or horizontal
(see~\cite{MirMor,MirMor1}). Note that, since a large fraction of
isogenies are descending, finding one of them is quite easy. However,
no known method can find horizontal or ascending isogenies without
using a trial-and-error approach. In this paper, we describe a first
solution to this open problem, which applies when the cardinality of
the curve is known, and propose a method that efficiently finds a
point $P$ of order $\ell$ that spans the kernel of an ascending (or
horizontal isogeny). Our approach relies on the computation of a small number of pairings.
We then show that our algorithms for traveling on the volcano are, in
many cases, faster than the ones from~\cite{Kohel} and~\cite{FouMor}.
In addition, we obtain a simple method that detects most curves on the crater of their
volcano. Until now, the only curves that were easily identified
were those on the floor of volcanoes.
Finally, we introduce an invariant for curves lying at the same level in the $\ell$-volcano.
In order to compute this invariant, we need to compute the group structure and a few pairings.
This paper is organized as follows: Sections~\ref{IsogenyVolcanoes} and~\ref{Pairing}
present definitions and properties of isogeny volcanoes and
pairings. Section~\ref{PreliminaryResults} explains our method to find
ascending or horizontal isogenies using pairing computations. Finally,
in Section~\ref{WalkingOnTheVolcano}, we use this method to improve the
algorithms for ascending a volcano, for walking on its crater and for
computing the $\ell$-adic valuation of the conductor of the endomorphism ring.

\section{Background on isogeny volcanoes}\label{IsogenyVolcanoes}
In this paper, we rely on some results from complex multiplication theory and on Deuring's lifting
theorems. We denote by $\Elld_d(\C)$ the set of $\C$-isomorphism classes of elliptic curves whose
endomorphism ring is the order $\mathcal{O}_{d}$, with discriminant $d<0$. In this setting,
there is an action of the class group of $\mathcal{O}_{d}$ on $\Elld_d(\C)$.
Let $E\in \Elld_{d}(\C)$, $\Lambda$ its corresponding lattice
and $\mathfrak{a}$ an $\mathcal{O}_d$-ideal. We have a canonical
homomorphism from $\C/\Lambda$ to $\C/\mathfrak{a}^{-1}\Lambda$
which induces an isogeny usually denoted by $E\rightarrow \hat{\mathfrak{a}}*E$. This action on $\Elld_{d}(\C)$ is transitive and free~\cite[Prop.~II.1.2]{SivAdvanced}. Moreover~\cite[Cor.~II.1.5]{SivAdvanced}, the degree of the application
$E\rightarrow \hat{\mathfrak{a}}*E$ is $N(\mathfrak{a})$, the norm of
the ideal $\mathfrak{a}$.

Let $\F_q$ be a finite field, with $q=p^r$ and $p$ a prime number. We denote by $\Elld_{d}(\F_q)$ the set of isomorphism classes of elliptic curves defined over $\F_q$, having endomorphism ring $\mathcal{O}_d$. From Deuring's theorems~\cite{Deuring}, if $p$ is a prime number that splits completely in the ring class field of $\mathcal{O}_d$, we get a bijection $\Elld_{d}(\C)\rightarrow \Elld_{d}(\F_q)$. Furthermore, the class group action in characteristic zero respects this bijection, and we get an action of the class group also on $\Elld_d(\F_q)$.

\subsection{Isogeny volcanoes}
Consider an elliptic curve $E$ defined over a finite field $\F_q$.
Let $\ell$ be a prime different from $\textrm{char}(\F_q)$ and $I:E\rightarrow E^{'}$ be an $\ell$-isogeny, i.e. an isogeny of degree~$\ell$. We denote by $\mathcal{O}_d$ and $\mathcal{O}_{d'}$ the
endomorphism rings of $E$ and $E'$, respectively.
As shown in~\cite{Kohel}, this means that $\mathcal{O}_d$ contains $\mathcal{O}_{d^{'}}$ or $\mathcal{O}_{d^{'}}$ contains $\mathcal{O}_d$ or the two endomorphism rings coincide.
If $\mathcal{O}_d$ contains $\mathcal{O}_{d^{'}}$, we say that $I$ is a \textit{descending} isogeny. Otherwise,
if $\mathcal{O}_d$ is contained in $\mathcal{O}_{d^{'}}$, we say that $I$ is a \textit{ascending} isogeny. If $\mathcal{O}_d$ and $\mathcal{O}_{d^{'}}$ are equal, then we call the isogeny \textit{horizontal}.
In his thesis, Kohel shows that horizontal isogenies exist only if the conductor of $\mathcal{O}_d$ is not divisible by $\ell$. Moreover, in this case there are exactly $\left (\frac{d}{\ell}\right )+1$ horizontal $\ell$-isogenies, where $d$ is the discriminant of $\mathcal{O}_d$. If $\left (\frac{d}{\ell} \right )=1$, then $\ell$ is split in $\mathcal{O}_d$ and the two horizontal isogenies correspond to the two actions $E\rightarrow \hat{\mathfrak{l}}*E$
and $E\rightarrow \hat{\bar{\mathfrak{l}}}*E$, where the two ideals
$\mathfrak{l}$ and $\bar{\mathfrak{l}}$ satisfy $(\ell)=\mathfrak{l}\,\bar{\mathfrak{l}}$.
In a similar way, if $\left (\frac{d}{\ell}\right )=0$, then $\ell$ is
ramified, i.e. $(\ell)=\mathfrak{l}^2$ and there is exactly one horizontal isogeny starting from $E$.
In order to describe the structure of the graph whose vertices are (isomorphism classes of) elliptic
curves with a fixed number of points and whose edges are $\ell$-isogenies, we recall
the following definition~\cite{Sutherland1}.
\begin{definition}
An $\ell$-volcano is a connected undirected graph with vertices
partitioned into levels $V_0,\ldots,V_h$, in which the subgraph on $V_0$
(the \textit{crater}) is a regular connected graph of degree at most 2
and
\begin{list}{}{\setlength{\topsep}{0in}}
\item[(a)] For $i>0$, each vertex in $V_i$ has exactly one edge leading to a vertex in $V_{i-1}$, and every edge not on the crater is of this form.
\item[(b)] For $i<h$, each vertex in $V_i$ has degree $\ell+1$.
\end{list}
\end{definition}
We call the level $V_h$ \textit{the floor} of the volcano. Vertices lying on the floor
have degree 1.
The following proposition~\cite{Sutherland1} follows essentially from~\cite[Prop. 23]{Kohel}.
\begin{proposition}\label{DefinitionVolcan}
Let $p$ be a prime number, $q=p^r$, and $d_{\pi}=t^2-4q$. Take $\ell\neq p$ another prime number. Let $G$ be the undirected graph with vertex
set  \rm{Ell}$_t(\F_q)$ and edges $\ell$-isogenies defined over $\F_q$. We denote by $\ell^h$ the largest power of $\ell$
dividing the conductor of $d_{\pi}$. Then the connected components of
$G$ that do not contain curves with $j$-invariant $0$ or $1728$ are $\ell$-volcanoes of height $h$ and for each component $V$, we have~:
\begin{list}{}{\setlength{\topsep}{0in}}
\item[(a)] The elliptic curves whose $j$-invariants lie in $V_0$ have endomorphism rings isomorphic to some $\mathcal{O}_{d_0}\supseteq \mathcal{O}_{d_{\pi}}$ whose conductor is not divisible by $\ell$.
\item[(b)] The elliptic curves whose $j$-invariants lie in $V_i$ have endomorphism rings isomorphic to $\mathcal{O}_{d_i}$, where $d_i=\ell^{2i}d_0$.
\end{list}
\end{proposition}
Elliptic curves are determined by their $j$-invariant, up to a
twist\footnote{For a definition of twists of elliptic curves, refer
to~\cite{Siv}.}.
Throughout the paper, we refer to a vertex in a volcano
by giving the curve or its $j$-invariant.

\subsection{Exploring the volcano}
Given a curve $E$ on an $\ell$-volcano, two methods are known to find its neighbours. The first method relies on the use of modular polynomials. The $\ell$-th \textit{modular polynomial}, denoted by $\Phi_{\ell}(X,Y)$ is a polynomial with integer coefficients. It satisfies the following property: given two elliptic curves $E$ and $E'$ with $j$-invariants $j(E)$ and $j(E')$ in $\F_q$, there is an
 $\ell$-isogeny from $E$ to $E'$ defined over $\F_q$, if and only if,
$\#E(\F_q)=\#E'(\F_q)$ and $\Phi_{\ell}(j(E),j(E'))=0$.
As a consequence, the curves related to $E$ via an $\ell$-isogeny can
be found by solving $\Phi_{\ell}(X,j(E))=0$. As stated
in~\cite{Schoof}, this polynomial\footnote{The case where the
  modular polynomial does not have any root corresponds to a degenerate
case of isogeny volcanoes containing a single curve and no $\ell$-isogenies.} may have $0$, $1$, $2$ or $\ell+1$ roots in $\F_q$. In order to find an edge on the volcano, it suffices to find a root $j'$ of this polynomial. Finally, if we need the equation of the curve $E'$ with $j$-invariant $j'$, we may use the formula in~\cite{Schoof}.


The second method to build $\ell$-isogenous curves constructs, given a point $P$  of order $\ell$ on $E$, the $\ell$-isogeny $I:E\rightarrow E'$ whose kernel $G$ is generated by $P$ using V\'elu's classical formulae~\cite{Velu} in an extension field $\F_{q^r}$.
To use this approach, we need the explicit coordinates of points of order $\ell$ on $E$.
We denote by $G_i$, $1\leq i\leq \ell+1$, the $\ell+1$ subgroups of order $\ell$ of $E$. Miret et al.~\cite{MirMor1} give the degree $r_i$ of the smallest extension field of $\F_q$ such that $G_i\subset \F_{q^{r_i}}$, $1\leq i\leq \ell+1$.
This degree is related to the order of $q$ in the group $\F_{\ell}^*$, that we denote by $\textrm{ord}_{\ell}(q)$.
\begin{proposition}
Let $E$ defined over $\F_q$ be an elliptic curve with $\kappa$ rational $\ell$-isogenies, with $\ell>2$. Let $G_i$, $1\leq i\leq \kappa$, be the kernels of these isogenies, and let $r_i$ be the minimum value for which $G_i\subset E(\F_{q^{r_i}})$.
\begin{list}{}{\setlength{\topsep}{0in}}
\item[(a)] If $\kappa=1$ then $r_1=\textrm{ord}_{\ell}(q)$ or $r_1=2\textrm{ord}_{\ell}(q)$.
\item[(b)] If $\kappa=\ell+1$ then either $r_i=\textrm{ord}_{\ell}(q)$ for all $i$, or $r_i=2\textrm{ord}_{\ell}(q)$ for all $i$.
\item[(c)] If $\kappa=2$ then $r_i|\ell-1$ for $i=1,2$.
\end{list}
\end{proposition}

In some cases, if the $\ell$-torsion is not defined over $\F_q$, it may be
preferable to replace the curve by its twist, if the $\ell$-torsion of the twist is defined over an extension field of smaller degree.
We also need the following corollary~\cite{MirMor1}.
\begin{corollary}\label{existencelTorsion}
Let $E/\F_q$ be an elliptic curve over $\F_q$ and $\tilde{E}$ its quadratic twist. If $E/\F_q$ has $1$ or $\ell+1$ rational $\ell$-isogenies,
then $\#E(\F_{q^{\textrm{ord}_{\ell}q}})$ or $\#\tilde{E}(\F_{q^{\textrm{ord}_{\ell}q}})$ is a multiple of $\ell$.
Moreover, if there are $\ell+1$ rational isogenies, then it is a multiple of $\ell^2$.
\end{corollary}
\subsection{The group structure of the elliptic curve on the volcano}
\begin{figure}
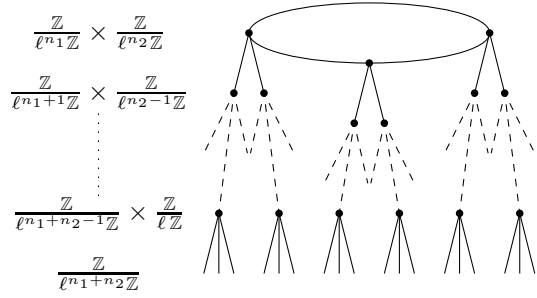

\begin{center}
\begin{pgfpicture}{-4cm}{1cm}{13cm}{4cm}
\pgfsetxvec{\pgfpoint{0.4cm}{0cm}}
\pgfsetyvec{\pgfpoint{0cm}{0.40cm}}
\pgfellipse[stroke]{\pgfxy(5,10)}{\pgfxy(4,0)}{\pgfxy(0,1)}
\pgfputat{\pgfxy(-4,10)}{\pgfbox[center,center]{$\frac{\Z}{\ell^{n_1}\Z}\times
\frac{\Z}{\ell^{n_2}\Z}   $}}
\pgfcircle[fill]{\pgfxy(1,10)}{0.05cm}
\pgfcircle[fill]{\pgfxy(9,10)}{0.05cm}
\pgfcircle[fill]{\pgfxy(5,9)}{0.05cm}
\pgfline{\pgfxy(1,10)}{\pgfxy(0.5,8)}
\pgfline{\pgfxy(1,10)}{\pgfxy(1.5,8)}
\pgfline{\pgfxy(9,10)}{\pgfxy(8.5,8)}
\pgfline{\pgfxy(9,10)}{\pgfxy(9.5,8)}
\pgfline{\pgfxy(5,9)}{\pgfxy(5.5,7)}
\pgfline{\pgfxy(5,9)}{\pgfxy(4.5,7)}
\pgfputat{\pgfxy(-4,8)}{\pgfbox[center,center]{$\frac{\Z}{\ell^{n_1+1}\Z}\times
\frac{\Z}{\ell^{n_2-1}\Z}   $}}
\pgfcircle[fill]{\pgfxy(0.5,8)}{0.05cm}
\pgfcircle[fill]{\pgfxy(1.5,8)}{0.05cm}
\pgfcircle[fill]{\pgfxy(8.5,8)}{0.05cm}
\pgfcircle[fill]{\pgfxy(9.5,8)}{0.05cm}
\pgfcircle[fill]{\pgfxy(4.5,7)}{0.05cm}
\pgfcircle[fill]{\pgfxy(5.5,7)}{0.05cm}

\pgfxyline(0,4)(0.5,2)
\pgfxyline(0,4)(0,2)
\pgfxyline(0,4)(-0.5,2)
\pgfxyline(2,4)(1.5,2)
\pgfxyline(2,4)(2,2)
\pgfxyline(2,4)(2.5,2)

\pgfxyline(8,4)(8.5,2)
\pgfxyline(8,4)(8,2)
\pgfxyline(8,4)(7.5,2)
\pgfxyline(10,4)(9.5,2)
\pgfxyline(10,4)(10,2)
\pgfxyline(10,4)(10.5,2)

\pgfxyline(4,4)(3.5,2)
\pgfxyline(4,4)(4,2)
\pgfxyline(4,4)(4.5,2)
\pgfxyline(6,4)(5.5,2)
\pgfxyline(6,4)(6,2)
\pgfxyline(6,4)(6.5,2)

\pgfsetdash{{0.1cm}{0.1cm}}{0.05cm}
\pgfxyline(0.5,8)(-0.5,6)
\pgfxyline(0.5,8)(0,4)
\pgfxyline(0.5,8)(0.95,6)
\pgfxyline(1.5,8)(1.05,6)
\pgfxyline(1.5,8)(2,4)
\pgfxyline(1.5,8)(2.5,6)

\pgfxyline(8.5,8)(7.5,6)
\pgfxyline(8.5,8)(8,4)
\pgfxyline(8.5,8)(8.95,6)
\pgfxyline(9.5,8)(9.05,6)
\pgfxyline(9.5,8)(10,4)
\pgfxyline(9.5,8)(10.5,6)

\pgfxyline(4.5,7)(3.5,5)
\pgfxyline(4.5,7)(4,4)
\pgfxyline(4.5,7)(4.95,5)
\pgfxyline(5.5,7)(5.05,5)
\pgfxyline(5.5,7)(6,4)
\pgfxyline(5.5,7)(6.5,5)

\pgfsetdash{{0.02cm}{0.1cm}}{0.05cm}
\pgfxyline(-4,7.5)(-4,4.5)

\pgfcircle[fill]{\pgfxy(0,4)}{0.05cm}
\pgfcircle[fill]{\pgfxy(2,4)}{0.05cm}
\pgfcircle[fill]{\pgfxy(4,4)}{0.05cm}
\pgfcircle[fill]{\pgfxy(6,4)}{0.05cm}
\pgfcircle[fill]{\pgfxy(8,4)}{0.05cm}
\pgfcircle[fill]{\pgfxy(10,4)}{0.05cm}
\pgfputat{\pgfxy(-4,4)}{\pgfbox[center,center]{$\frac{\Z}{\ell^{n_1+n_2-1}\Z}\times
\frac{\Z}{\ell\,\Z}   $}}
\pgfputat{\pgfxy(-4,2)}{\pgfbox[center,center]{$\frac{\Z}{\ell^{n_1+n_2}\Z} $}}

\end{pgfpicture}
\caption{A regular volcano}
\label{RegularVolcano}
\end{center}
\end{figure}
Lenstra~\cite{Lenstra} relates the group structure of an elliptic
curve to its endomorphism ring by proving that
$E(\F_q)\simeq \mathcal{O}_E/(\pi-1)$
as $\mathcal{O}_E$-modules.
It is thus natural to see how this structure relates to the isogeny volcano.
From Lenstra's equation, we can deduce that $E(\F_q)\simeq
\Z/M\Z\times \Z/N\Z$, for some positive integers $N$ and $M$ with $N|M$. We denote by $g$ the conductor of $\Z[\pi]$ and we write $\pi=a+g\omega$, with:
$$a=
\left \{\begin{array}{l}
(t-g)/2 \\
t/2
\end{array}
\right.\mbox{~and~}\omega=\left \{\begin{array}{ll}
\frac{1+\sqrt{d_K}}{2}& \mbox{~~if}\,\, d_K\equiv 1\phantom{,3}\pmod{4}\\
\sqrt{\frac{d_K}{4}} &\mbox{~~if}\,\, d_K\equiv 0 \pmod{4}
\end{array}
\right.
$$
where $d_K$ is the discriminant of the quadratic imaginary field containing $\mathcal{O}_E$. Note that $N$ is maximal such that $E[N]\subset E(\F_q)$ and by~\cite[Lemma 1]{Ruck1} we get that $N=\textrm{gcd}(a-1,g/f)$, with $f$ the conductor of $\textrm{End}(E)$.
Note moreover that $N|M$, $N|(q-1)$ and $MN=\#E(\F_q).$
This implies that on an $\ell$-volcano the group structure of all the curves in a given level is the same.

\noindent
Let $E$ be a curve on the isogeny volcano such that
$v_{\ell}(N)<v_{\ell}(M)$. As explained in~\cite{MirMor} (in the case
$\ell=2$, but the result is general), $a$ is such that
$\displaystyle v_{\ell}(a-1)\geq \min\left\{v_{\ell}(g),v_{\ell}(\#E(\F_q))/2\right\}.$

\noindent Since $N=\gcd(a-1,g/f)$ and $v_{\ell}(N)\leq v_{\ell}(\#E(\F_q))/2$, it follows that $v_{\ell}(N)=v_{\ell}(g/f)$. As we descend, the valuation at $\ell$ of the conductor $f$ increases by $1$ at each level (by Proposition~\ref{DefinitionVolcan}b).
This implies that the $\ell$-valuation of $N$ for
curves at each level decreases by $1$ and is equal to $0$ for curves lying on the floor.
Note that if $v_{\ell}(\#E(\F_q))$ is even and the height $h$ of the volcano is greater than $v_{\ell}(\#E(\F_q))$, the structure of the $\ell$-torsion group is unaltered from the crater down to the level $h-v_{\ell}(\#E(\F_q))/2$. From this level down, the structure of the $\ell$-torsion groups starts changing as explained above. In the sequel, we call this
level the \textit{first stability level.}\footnote{Miret et al. call it simply \textit{the stability level}.}
A volcano with first stability level
equal to $0$, i.e. on the crater, is called \textit{regular} (see Figure~\ref{RegularVolcano}).
\vspace{0.1 cm}

\noindent
\textbf{Notations.}
Let $n\geq 0$. We denote by $E[\ell^n]$ the $\ell^n$-torsion subgroup, i.e. the subgroup of points of order dividing $\ell^n$ on the curve $E$, by $E[\ell^n](\F_{q^k})$ the subgroup of points of order dividing $\ell^n$ defined over an extension field of $\F_q$ and by $E[\ell^{\infty}](\F_{q})$ the $\ell$-Sylow subgroup of $E(\F_q)$.
\vspace{0.1 cm}

\section{Background on pairings}\label{Pairing}

Let $E$ be an elliptic curve defined over some finite field $\F_q$, $m$ an integer such that $m|\#E(\F_q)$.
 Let $k$ be the embedding degree, i.e. the smallest integer such that $m|q^k-1$. Let $P\in E[m](\F_{q^k})$ and $Q\in E(\F_{q^k})/mE(\F_{q^k})$. Let $f_{m,P}$ be the function whose divisor\footnote{For background on divisors, see~\cite{Siv}.}
is $m(P)-m(O)$, where $O$ is the point at infinity of the curve $E$. Take $R$ a random point in $E(\F_{q^k})$ such that the support of the divisor $D=(Q+R)-(R)$ is disjoint from the support of $f_{m,P}$.
Then we can define the Tate pairing as follows:
\begin{eqnarray*}
t_m:E[m](\F_{q^k})\times E(\F_{q^k})/mE(\F_{q^k})& \rightarrow & \F_{q^k}^*/(\F_{q^k}^*)^m\\
(P,Q) & \rightarrow & f_{m,P}(Q+R)/f_{m,P}(R).
\end{eqnarray*}

The Tate pairing is a bilinear non-degenerate map, i.e. for
all $P\in E[m](\F_{q^k})$ different from $O$ there is a $Q\in
E(\F_{q^k})/mE(\F_{q^k})$ such that $T_m(P,Q)\neq 1$. The output of the
pairing is only defined up to a coset of $(\F_{q^k}^*)^m$. However,
for implementation purposes, it is useful to have a uniquely defined
value and to use the \textit{reduced} Tate pairing, i.e.
$\displaystyle T_m(P,Q)=t_m(P,Q)^{(q-1)/m}\in \mu_m,$
where $\mu_m$ denotes the group of $m$-th roots of unity.
Pairing computation can be done in $O(\log{m})$ operations in $\F_q$ using Miller's
algorithm~\cite{JC:Miller04}. For more details and properties of pairings, the reader can refer to~\cite{Frey}. Note that in the recent years, in view of cryptographic applications, many
implementation techniques have been developed and pairings on elliptic curves can be computed very efficiently\footnote{See~\cite{GrabherGP08} for a fast recent implementation.}.\\
In the remainder of this paper we assume that the embedding degree is always $1$, i.e. $m|q-1$.
We will denote by $k$ a different integer.
Suppose now that $m=\ell^n$, with $n\geq 1$ and $\ell$ prime.
Now let $P$ and $Q$ be two $\ell^n$-torsion points on $E$.
We define the following symmetric pairing~\cite{JouNgu}
\begin{eqnarray}
S(P,Q)=(T_{\ell^n}(P,Q)\,T_{\ell^n}(Q,P))^{\frac{1}{2}}.
\end{eqnarray}
Note that for any point $P$, $T_{\ell^n}(P,P)=S(P,P)$. In the remainder of this paper, we call $S(P,P)$
\textit{the self-pairing} of $P$. We focus on the case where the
pairing $S$ is non-constant.
Suppose now that $P$ and $Q$ are two linearly independent $\ell^{n}$-torsion points. Then all $\ell^n$-torsion points $R$ can be expressed as $R=aP+bQ$.
Using bilinearity and symmetry of the $S$-pairing, we get
\begin{eqnarray*}
\log(S(R,R))=a^2\log(S(P,P))+2ab\,\log(S(P,Q))+b^2\log(S(Q,Q)) \pmod{\ell^n},
\end{eqnarray*}
where log is a discrete logarithm function in $\mu_{\ell^n}$.
We denote by $k(E)$ the largest integer such that the polynomial
\begin{eqnarray}\label{PairingPolynomial}
\PP(a,b)=a^2\log(S(P,P))+2ab\,\log(S(P,Q))+b^2\log(S(Q,Q))
\end{eqnarray}
is identically zero modulo $\ell^{n-k(E)-1}$ and nonzero modulo $\ell^{n-k(E)}$. Obviously,
since $S$ is non-constant we have $0\leq k(E)< n$. Dividing by $\ell^{n-k(E)-1}$, we may thus view $\PP$ as a polynomial in $\F_{\ell}[a,b]$. When we
want to emphasize the choice of $E$ and $\ell^n$, we write
$\PP_{E,\ell^n}$ instead of $\PP$.

Since $\PP$ is a non-zero quadratic polynomial, it
has at most two homogeneous roots, which means
that from all the $\ell+1$ subgroups of $E[\ell^n]/E[\ell^{n-1}]\simeq (\Z/\ell\Z)^2$, at most $2$
have self-pairings in $\mu_{\ell^{k(E)}}$ (see also~\cite{JouNgu}).
In the remainder of this paper, we denote by $N_{E,\ell^n}$ the number
of zeros of $\PP_{E,\ell^n}$.
Note that this number does not depend on the choice of the two generators $P$ and $Q$
of the $\ell^n$-torsion subgroup $E[\ell^n]$.
Moreover, we say that a $\ell^n$-torsion point $R$ has \textit{degenerate self-pairing} if $T_{\ell^n}(R,R)$ is a
$\ell^{k(E)}$-th root of unity and that $R$ has
\textit{non-degenerate self-pairing} if $T_{\ell^n}(R,R)$ is a
primitive $\ell^{k(E)+1}$-th root of unity. Also, if
$T_{\ell^n}(R,R)$ is a primitive $\ell^n$-th root of unity, we say
that $R$ has \textit{primitive self-pairing}.
\section{Determining directions on the volcano}\label{PreliminaryResults}

In this section, we explain how we can distinguish between different directions on the volcano by making use of pairings.
Given a point $P\in E[\ell^n](\F_q)$, we also need to know the degree
of the smallest extension field containing an $\ell^{n+1}$-torsion point
such that $\ell\tilde{P}=P$. The following result is taken from~\cite{Fouquet}.
\begin{proposition}\label{shapeVolcano}
Let $\ell>2$ and $E/\F_q$ be an elliptic curve which lies on an $\ell$-volcano whose height $h(V)$ is different from $0$.
Then the height of $V'$, the $\ell$-volcano of the curve $E/\F_{q^s}$ is
$ \displaystyle h(V')=h(V)+v_{\ell}(s).$
\end{proposition}
From this proposition, it follows easily that if the structure of the subgroup $E[\ell^{\infty}](\F_q)$ on the curve $E$ is $\Z/\ell^{n_1}\Z \times \Z/\ell^{n_2}\Z$, then the smallest extension $K$ of $\F_q$ such that $E[\ell^{\infty}](K)$ is not isomorphic to $E[\ell^{\infty}](\F_q)$ is $\F_{q^\ell}$.

\begin{proposition}\label{ChangeStructure}
Let $\ell>2$ and $E/\F_q$ be an elliptic curve with $E[\ell^{\infty}](\F_q)\simeq \Z/\ell^{n_1}\Z\times \Z/\ell^{n_2}\Z$, with $n_2\geq 1$. Then
$$E[\ell^{\infty}](\F_{q^{\ell}})\simeq \Z/\ell^{n_1+1}\Z \times \Z/\ell^{n_2+1}\Z.$$
\end{proposition}
\begin{proof}
  Note that $E$ lies on an $\ell$-volcano $V/\F_q$ of height at least $n_2$. We consider a curve $E'$ lying on the floor of $V/\F_q$ such that there is a descending path of isogenies between $E$ and $E'$. Obviously, we have $E'[\ell^{\infty}](\F_q)\simeq \Z/\ell^{n_1+n_2}\Z$.
By Proposition~\ref{shapeVolcano}, $V/\F_{q^{\ell}}$ has one extra down level, which means that the curve $E'$
is no longer on the floor, but on the level just above the floor. Consequently, we have that $E'[\ell]\subset E'(\F_{q^{\ell}})$ and, moreover, $E'[\ell^{\infty}](\F_{q^{\ell}})\simeq \Z/\ell^{n_1+n_2+\Delta}\Z \times  \Z/\ell \Z$.

We now show that $\Delta=1$. Note first that $\ell^{n_2}|q-1$ and that $v_{\ell}(q^{\ell}-1)=v_{\ell}(q-1)+1$. We denote by $P$ a point of order $\ell^{n_1+n_2+\Delta}$ on the curve $E'/\F_{q^{\ell}}$.
Then, without restraining the generality, we may assume that
\begin{eqnarray}\label{tateDegenerate}
T_{\ell^{n_2}}^{(\F_q)}(\ell^{n_1+\Delta}P,\ell^{\Delta}P)=f_{\ell^{n_2},\ell^{n_1+\Delta}P}(\ell^{\Delta}P)^{\frac{q-1}{\ell^{n_2}}}\in \mu_{\ell^{n_2}}\backslash \mu_{\ell^{n_2-1}},
\end{eqnarray}
and
\begin{eqnarray}\label{tateDegenerate1}
T_{\ell^{n_2+1}}^{(\F_{q^{\ell}})}(\ell^{n_1+\Delta-1}P,P)=f_{\ell^{n_2+1},\ell^{n_1+\Delta-1}P}(P)^{\frac{q^{\ell}-1}{\ell^{n_2+1}}}\in \mu_{\ell^{n_2+1}}\backslash \mu_{\ell^{n_2}}.
\end{eqnarray}
By using the bilinearity of the pairing and the fact that $f_{\ell^{n_2+1},R}=f_{\ell^{n_2},R}^{\ell}$ for a point of order $\ell^{n_2}$ (up to a constant), we get from Equation~\eqref{tateDegenerate}
 $$f_{\ell^{n_2},\ell^{n_1+\Delta}P}(P)^{\ell\frac{q^{\ell}-1}{\ell^{n_2+1}}}\in \mu_{\ell^{n_2}}\backslash \mu_{\ell^{n_2-1}}.$$
By using Equality~(\ref{tateDegenerate}), this is true if and only if $\Delta=1$.
By ascending on the volcano from $E'$ to $E$, we deduce that the structure of the $\ell$-torsion of $E$ over
$\F_{q^{\ell}}$ is necessarily
\begin{eqnarray*}
E[\ell^{\infty}](\F_{q^{\ell}})\simeq \Z/\ell^{n_1+1}\Z\times \Z/\ell^{n_2+1}\Z.
\end{eqnarray*}
\end{proof}

\begin{remark}\label{extVolcano2}
If $\ell=2$, the only problematic case is when the $\ell$-adic valuation of the conductor of $\Z[\pi]$ is $1$. In all the other cases, the volcano gets exactly one extra level over $\F_{q^2}$ (see~\cite{Fouquet}). Reasoning as in the proof of Proposition~\ref{ChangeStructure}, we get that for a curve $E$ on a $2$-volcano of height at least $2$ such that $E[2^{\infty}](\F_q)\simeq \Z/2^{n_1}\Z\times \Z/2^{n_2}\Z$, the $2$-Sylow group structure over $\F_{q^2}$ is
$$E[2^{\infty}](\F_{q^2})\simeq \Z/2^{n_1+\Delta}\Z\times \Z/2^{n_2+1}\Z.$$
However, the following example shows that when $\ell=2$, $\Delta$ is not always $1$.
\end{remark}

\begin{example}
Let $E$ be an an elliptic curve defined over $\F_q$ with $q=257$ given by the equation
$$y^2=x^3+206x^2+221x+33.$$
Then $E[2^{\infty}][\F_q]\simeq \Z/2\Z\times \Z/2\Z$ and $E[2^{\infty}][\F_{q^2}]\simeq \Z/2^4\Z\times \Z/2^2\Z$.
\end{example}

\begin{remark}
We note that in the general context of ordinary abelian varieties, Freeman and Lauter~\cite{FreLau} proved that if the $\ell^n$-torsion is defined over a finite field $\F_q$, then the $\ell^{n+1}$-torsion is defined over $\F_{q^{\ell}}$.
\end{remark}
\noindent
We give some lemmas explaining the relations between pairings on two isogenous curves.

\begin{lemma}\label{TatePairing}
Suppose $E/\F_q$ is an elliptic curve and $P,Q$ are points in $E(\F_q)$ of order $\ell^n$, $n\geq 1$.
Denote by $\tilde{P},\tilde{Q}\in E[\bar{\F}_q]$ two points such that $\ell\tilde{P}=P$ and $\ell\tilde{Q}=Q$.
Suppose that $\ell^n|q-1$. Then we have the following relations for the Tate pairing
\begin{list}{}{\setlength{\topsep}{0in}}
\item[(a)] If $\tilde{P},\tilde{Q}\in E[\F_q]$, then
$
\displaystyle
T_{\ell^{n+1}}(\tilde{P},\tilde{Q})^{\ell^2}=T_{\ell^n}(P,Q).
$
\item[(b)] Suppose $\ell\geq 3$. If $\tilde{Q}\in E[\F_{q^{\ell}}]\backslash E[\F_q]$, then
$
\displaystyle
T_{\ell^{n+1}}(\tilde{P},\tilde{Q})^{\ell}=T_{\ell^n}(P,Q).
$
\item[(c)] Let $\ell=2$ and $\tilde{Q}\in E[\F_{q^{2}}]\backslash E[\F_q]$. Then
$
\displaystyle
T_{2^{n+1}}(\tilde{P},\tilde{Q})^{\ell}=T_{2^n}(P,Q)T_{2^n}(P,T),
$
where $T$ is a point of order $2$.
\end{list}
\end{lemma}
\begin{proof}
(a) By writing down the divisors of the functions $f_{\ell^{n+1},\tilde{P}}$, $f_{\ell^n,\tilde{P}}$, $f_{\ell^n,P}$, one can easily
check that
$$f_{\ell^{n+1},\tilde{P}}=(f_{\ell,\tilde{P}})^{\ell^n}\cdot f_{\ell^n,P}.$$
We evaluate these functions at some points $Q+R$ and $R$ (where $R$ is carefully chosen) and raise the equality to the power
$(q-1)/\ell^n$.\\
(b) Due to the equality on divisors $\textrm{div}(f_{\ell^{n+1},P})=\textrm{div}(f_{\ell^n,P}^{\ell})$, we have
\begin{eqnarray*}
T_{\ell^{n+1}}(\tilde{P},\tilde{Q})^\ell=T_{\ell^n}^{(\F_{q^{\ell}})}(P,\tilde{Q}),
\end{eqnarray*}
where $T_{\ell^n}^{(\F_{q^{\ell}})}$ is the $\ell^n$-Tate pairing for $E$ defined over $\F_{q^{\ell}}$.
It suffices then to show that
$\displaystyle T_{\ell^n}^{(\F_{q^{\ell}})}(P,\tilde{Q})=T_{\ell^n}(P,Q).$
We have
\begin{eqnarray}
\nonumber  T_{\ell^n}^{(\F_{q^{\ell}})}(P,\tilde{Q})&=&f_{\ell^n,P}([\tilde{Q}+R]-[R])^\frac{(1+q+\dots+q^{\ell-1})(q-1)}{\ell^n}\\\nonumber&=&
f_{\ell^n,P}((\tilde{Q}+R)+(\pi(\tilde{Q})+R)+(\pi^2(\tilde{Q})+R)+\dots\\\label{Frobenius}&+&(\pi^{\ell-1}(\tilde{Q})+R)-\ell(R))^{\frac{(q-1)}{\ell^n}}
\end{eqnarray}
where $R$ is a random point defined over $\F_q$.
It is now easy to see that for $\ell\geq 3$,
\begin{eqnarray}\label{frobeniussum}
\tilde{Q}+\pi(\tilde{Q})+\pi^2(\tilde{Q})+\ldots+\pi^{\ell-1}(\tilde{Q})=\ell\tilde{Q}=Q,
\end{eqnarray}
because $\pi(\tilde{Q})=\tilde{Q}+T$, where $T$ is a point of order $\ell$.
By applying Weil's reciprocity law~\cite[Ex. II.2.11]{Siv}, it follows that the equation~\eqref{Frobenius} becomes:
\begin{eqnarray}\label{finally}
T_{\ell^n}^{(\F_{q^{\ell}})}(P,\tilde{Q})&=&\left( \frac{f_{\ell^n,P}(Q+R)}{f_{\ell^n,P}(R)}\right )^{\frac{q-1}{\ell^n}}f((P)-(O))^{q-1},
\end{eqnarray}
where $f$ is such that $\textrm{div}(f)=(\tilde{Q}+R)+(\pi(\tilde{Q})+R)+(\pi^2(\tilde{Q})+R)+...+(\pi^{\ell-1}(\tilde{Q})+R)-(Q+T+R)-(\ell-1)(R)$. Note that this divisor is $\F_q$-rational, so $f((P)-(O))^{q-1}=1$. This concludes the proof.\\
(c) The sum at~\eqref{frobeniussum} becomes
\begin{eqnarray}
\tilde{Q}+\pi(\tilde{Q})=Q+T,
\end{eqnarray}
where $T$ is a point of order $2$.
Consequently, we have an equation similar to equation~\eqref{finally}
\begin{eqnarray*}
T_{2^n}^{(\F_{q^{2}})}(P,\tilde{Q})=\left( \frac{f_{2^n,P}(Q+T+R)}{f_{2^n,P}(R)}\right )^{\frac{q-1}{2^n}}f((P)-(O))^{q-1},
\end{eqnarray*}
where $f$ is such that $\textrm{div}(f)=(\tilde{Q}+R)+(\pi(\tilde{Q})+R)-(Q+T+R)-(R)$. We know that
$f$ is rational, hence $f((P)-(O))^{q-1}=1$. We conclude that
$$T_{2^{n+1}}(\tilde{P},\tilde{Q})^{2}=T_{2^n}(P,Q)T_{2^n}(P,T).$$
\end{proof}
\begin{lemma}\label{ImportantLemma}
Let $\phi:E\rightarrow E'$ be a separable isogeny defined over a finite field $\F_q$, $\ell\in \Z$ such that $\ell|q-1$.
\begin{list}{}{\setlength{\topsep}{0in}}
\item[(a)]  Denote by $d$ the degree of the isogeny and by $P$ an $\ell$-torsion on the curve $E$ such
that $\phi(P)$ is an $\ell$-torsion point on $E'$, and $Q$ a point on $E$. Then we have
$$
\displaystyle
T_{\ell}(\phi(P),\phi(Q))=T_{\ell}(P,Q)^d.
$$
\item[(b)] Let $\phi:E\rightarrow E'$ be a separable isogeny of degree $\ell$ defined over $\F_q$, $P$ a $\ell\ell'$-torsion point such that
$\textrm{Ker}\,\,\phi=\langle \ell'P \rangle$ and $Q$ a point on the curve $E$.
Then we have
$$
\displaystyle
T_{\ell}(\phi(P),\phi(Q))=T_{\ell\ell'}(P,Q)^{\ell}.
$$
\end{list}
\end{lemma}
\begin{proof}
 (a) We have
\begin{eqnarray*}
(\phi)^*(f_{\ell,\phi(P)})&=&\ell \sum_{K\in \textrm{Ker}\phi} ((P+K)-(K))=\ell \sum_{K\in \textrm{Ker}\phi} ((P)-(O))\\&+&\textrm{div}\left
(\left(\prod_{K\in \textrm{Ker}\phi } \frac{l_{K,P}}{v_{K+P}}\right)^{\ell }\right ),
\end{eqnarray*}
where $l_{K,P}$ is the straight line passing through $K$ and $P$ and $v_{K+P}$ is the vertical line passing through $K+P$.
It follows that for some point $S$ on $E$
\begin{eqnarray*}
f_{\ell,\phi(P)}\circ \phi(S)=f_{\ell,P}^d(S)\left (\prod _{K\in \textrm{Ker}\phi}\frac{l_{K,P}(S)}{v_{K+P}(S)}\right )^{\ell}.
\end{eqnarray*}
We obtain the desired formula by evaluating the equality above at two points carefully chosen $Q+R$ and $R$, and then by raising to the power $\frac{q-1}{\ell}$.\\
\noindent
(b) This time we have
\begin{eqnarray*}
(\phi)^*(f_{\ell',\phi(P)})&=& \ell' \sum_{K\in \textrm{Ker}\phi} ((P+K)-(K))= \ell' \sum_{K\in \textrm{Ker}\phi} ((P)-(O))\\&&+\textrm{div}\left
(\left(\prod_{K\in \textrm{Ker}\phi } \frac{l_{K,P}}{v_{K+P}}\right)^{\ell'}\right ),
\end{eqnarray*}
Since $\#\textrm{Ker}\phi=\ell$, we get
\begin{eqnarray*}
f_{\ell',\phi(P)}\circ \phi(Q)=f_{\ell\ell',P}(Q)\left (\prod _{K\in \textrm{Ker}\phi}\frac{l_{K,P}(Q)}{v_{K+P}(Q)}\right )^{\ell'}.
\end{eqnarray*}
We raise this equality to the power $\frac{q-1}{\ell'}$ and get the announced result.
\end{proof}

\begin{proposition}\label{descendancyProof}
Let $E$ be an elliptic curve defined a finite field $\F_q$ and assume that $E[\ell^{\infty}](\mathbb{F}_q)$ is isomorphic
to $\Z/\ell^{n_1}\Z\times \Z/\ell^{n_2}\Z$ (with $n_1 \geq n_2\geq 1$). Suppose that there is a $\ell^{n_2}$-torsion point $P$ such that
$T_{\ell^{n_2}}(P,P)$ is a primitive $\ell^{n_2}$-th root of unity. Then the $\ell$-isogeny whose kernel is generated by $\ell^{n_2-1}P$ is descending.
Moreover, the curve $E$ does not lie above the first stability level of the corresponding $\ell$-volcano.
\end{proposition}
\begin{proof}
Let $I_1:E\rightarrow E_1$ be the isogeny whose kernel is generated by $\ell^{n_2-1}P$ and suppose this isogeny is ascending or horizontal.
This means that $E_1[\ell^{n_2}]$ is defined over $\F_q$. Take $Q$
another $\ell^{n_2}$-torsion point on $E$, such that
$E[\ell^{n_2}]=\langle P,Q \rangle$ and denote
by $Q_1=I_1(Q)$. One can easily check that the dual of $I_1$ has kernel generated by $\ell^{n_2-1}Q_1$. It follows that there is a point $P_1\in E_1[\ell^{n_2}]$ such that $P=\hat{I_1}(P_1)$. By Lemma~\ref{ImportantLemma} this means that $T_{\ell}(P,P)\in \mu_{\ell^{n_2-1}}$, which is false. This proves not only that the isogeny is descending, but also that the structure of the $\ell$-torsion is different at the level of $E_1$. Hence $E$ cannot be above the stability level.
\end{proof}

\begin{proposition}\label{existenceNondegenerate}
Let $E/\F_q$ be a curve which lies in an $\ell$-volcano and on the first stability level. Suppose $E[\ell^{\infty}](\F_q)\simeq \Z/\ell^{n_1}\Z\times \Z/\ell^{n_2}\Z$, $n_1\geq n_2\geq 1$.
\begin{itemize} \item[(a)] Suppose $\ell \geq 3$. Then there is at least one $\ell^{n_2}$-torsion point $E(\F_q)$ with primitive self-pairing.
\item[(b)] If $\ell=2$ and the height of the volcano is greater than 1, then there is at least one $\ell^{n_2}$-torsion point $E(\F_q)$ with primitive self-pairing.
\end{itemize}
\end{proposition}
\begin{proof} (a) Let $P$ be a $\ell^{n_1}$-torsion point and $Q$ be a $\ell^{n_2}$-torsion point such that $\{P,Q\}$ generates $E[\ell^{\infty}](\F_q)$.\\
\textsl{Case 1.} Suppose $n_1\geq n_2\geq 2$. Let $E \stackrel{ I_1 }{\longrightarrow} E_1$ be a descending $\ell$-isogeny and denote by $P_1$ and $Q_1$ the $\ell^{n_1+1}$ and $\ell^{n_2-1}$-torsion points generating $E_1[\ell^{\infty}](\F_p)$.
Moreover, without loss of generality, we may  assume that $I_1(P)=\ell P_1$ and $I_1(Q)=Q_1$.
If $T_{\ell^{n_2-1}}(Q_1,Q_1)$ is a primitive $\ell^{n_2-1}$-th root of unity, $T_{\ell^{n_2}}(Q,Q)$ is a primitive $\ell^{n_2}$-th root of unity
by Lemma~\ref{ImportantLemma}. If not, from the non-degeneration of the pairing, we deduce that $T_{\ell^{n_2-1}}(Q_1,P_1)$ is a primitive $\ell^{n_2-1}$-th
root of unity, which means that $T_{\ell^{n_2-1}}(Q_1,\ell P_1)$ is a $\ell^{n_2-2}$-th primitive root of unity.
By applying Lemma~\ref{ImportantLemma}, we get $T_{\ell^{n_2}}(Q,P)\in \mu_{\ell^{n_2-1}}$ at best. It follows that $T_{\ell^{n_2}}(Q,Q)\in \mu_{\ell^{n_2}}$
by the non-degeneracy of the pairing.\\
\textsl{Case 2.} If $n_2=1$, then consider the volcano defined over the extension field $\F_{q^{\ell}}$. There is a $\ell^2$-torsion point $\tilde{Q}\in E(\F_{q^{\ell}})$ with $Q=\ell\tilde{Q}$. We obviously have $\ell^2|q^{\ell}-1$ and from Lemma~\ref{TatePairing}, we get $T_{\ell^2}(\tilde{P},\tilde{P})^{\ell}=T_{\ell}(P,P)$.
By applying Case 1, we get that $T_{{\ell}^2}(\tilde{P},\tilde{P})$ is a primitive $\ell^2$-th root of unity, so $T_{\ell}(P,P)$ is a primitive $\ell$-th root of unity.\\
(b) If $n_2>1$, the proof is similar to that of (a) Case 1. Suppose now that $n_2=1$. Since $4|\#E(\F_q)$, we have $q+1-t\equiv 0\,\, \textrm{mod}\,\,4$. Then $t^2-4q\equiv (q-1)^2\,\,\textrm{mod}\,\,4$. We deduce that $E$ lies on a $2$-volcano with height greater than $1$ if and only if $q\equiv 1 \,\,\textrm{mod}\,\,4$.
Let $E'$ be a curve on the floor of the $\ell$-volcano such that there is a $2$-ascending isogeny $I:E'\rightarrow E$. The fact that $4|q-1$ implies that the $4$-th Tate pairing is well-defined over $\F_q$ and non-degenerate. We have $E'[2^{\infty}](\F_q)\simeq \Z/4\Z$ and thus there is a point $P\in E'[4](\F_q)$ such that $T_4(P,P)\in \mu_4^{*}$ and that $I(2P)=0$. By applying Lemma~\ref{ImportantLemma}, we get that $$T_2(I(P),I(P))\in \mu_2^{*}.$$
\end{proof}


\noindent
We now make use of a result on the representation of ideal classes of orders in imaginary quadratic fields. This is Corollary 7.17 from~\cite{Cox}.
\begin{lemma}\label{CoxLemma}
Let $\mathcal{O}$ be an order in an imaginary quadratic field. Given a
nonzero integer $M$, then every ideal class in
$\textrm{Cl}(\mathcal{O})$ contains a proper $\mathcal{O}$-ideal whose
norm is relatively prime to $M$.
\end{lemma}
\begin{proposition}\label{Constancy}
We use the notations and assumptions from Proposition~\ref{DefinitionVolcan}. Furthermore, we assume that for all curves $E_i$ lying at a fixed level $i$ in $V$ the curve structure is $\Z/\ell^{n_1}\Z \times \Z/\ell^{n_2}\Z$, with $n_1\geq n_2\geq 1$. The value of $N_{E_i,\ell^{n_2}}$, the number of zeros of the polynomial defined at~\eqref{PairingPolynomial}, is constant for all curves lying at level $i$ in the volcano.
\end{proposition}
\begin{proof} Let $E_1$ and $E_2$ be two curves lying at level $i$ in the volcano $V$. Then by Proposition~\ref{DefinitionVolcan}
they both have endomorphism ring isomorphic to some order $\mathcal{O}_{d_i}$. Now by taking into account the fact that the action of $\textrm{Cl}(\mathcal{O}_{d_i})$ on $\Elld_{d_i}(\F_q)$ is transitive, we consider an isogeny $\phi:E_1\rightarrow E_2$ of degree $\ell_1$. By applying Lemma~\ref{CoxLemma}, we may assume that $(\ell_1,\ell)=1$. Take now $P$ and $Q$ two independent $\ell^{n_2}$-torsion points on $E_1$ and denote by $\PP_{E_1,\ell^{n_2}}$ the quadratic polynomial corresponding to the $\ell^{n_2}$-torsion on $E_1$ as in~\ref{PairingPolynomial}.
We use Lemma~\ref{ImportantLemma} to compute $S(\phi(P),\phi(P))$, $S(\phi(P),\phi(Q))$ and $S(\phi(Q),\phi(Q))$ and deduce that a polynomial $\PP_{E_2,\ell^{n_2}}(a,b)$ on the curve $E_2$ computed from $\phi(P)$ and $\phi(Q)$ is such that
\begin{eqnarray*}
\PP_{E_1,\ell^{n_2}}(a,b)=\PP_{E_2,\ell^{n_2}}(a,b).
\end{eqnarray*}
This means that $N_{E_1,\ell^{n_2}}$ and $N_{E_2,\ell^{n_2}}$ coincide, which concludes the proof.
Moreover, we have showed that the value of $k(E_1)=k(E_2)$.
\end{proof}

\begin{proposition}\label{SmartNavigating}
Let $E$ be an elliptic curve defined a finite field $\F_q$ and let $E[\ell^{\infty}](\mathbb{F}_q)$ be isomorphic to $\Z/\ell^{n_1}\Z\times \Z/\ell^{n_2}\Z$ with $n_1 \geq n_2\geq 1$. Suppose $N_{E,\ell^{n_2}}\in \{1,2\}$ and let $P$ be a $\ell^{n_2}$-torsion point with degenerate self-pairing.
Then the $\ell$-isogeny whose kernel is generated by $\ell^{n_2-1}P$ is either ascending or horizontal. Moreover, for any $\ell^{n_2}$-torsion point $Q$ whose self-pairing is non-degenerate, the isogeny with kernel spanned by $\ell^{n_2-1}Q$ is descending.
\end{proposition}
\begin{proof}
\textit{Case 1.}
Suppose $T_{\ell^{n_2}}(P,P)\in \mu_{\ell^{k(E)}}$, $k(E)\geq 1$ and that $T_{\ell^{n_2}}(Q,Q)\in \mu_{\ell^{k(E)+1}}\backslash \mu_{\ell^{k(E)}}$. Denote by $I_1:E\rightarrow E_1$ the isogeny whose kernel
 is generated by $\ell^{n_2-1}P$ and $I_2:E\rightarrow E_2$ the isogeny whose kernel is generated by $\ell^{n_2-1}Q$. By repeatedly applying Lemmas~\ref{TatePairing} and
\ref{ImportantLemma}, we get the following relations for points generating the $\ell^{n_2-1}$-torsion on $E_1$ and $E_2$:
\begin{eqnarray*}
T_{\ell^{n_2-1}}(I_1(P),I_1(P))&\in & \mu_{\ell^{k(E)-1}}  ,\,\,T_{\ell^{n_2-1}}(\ell I_1(Q),\ell I_1(Q))\in \mu_{\ell^{k(E)-2}}\backslash \mu_{\ell^{k(E)-3}}\\
T_{\ell^{n_2-1}}(\ell I_2(P),\ell I_2(P))&\in & \mu_{\ell^{k(E)-3}}  ,\,\,T_{\ell^{n_2-1}}(I_2(Q),I_2(Q))\in \mu_{\ell^{k(E)}}\backslash \mu_{\ell^{k(E)-1}}
\end{eqnarray*}
with the convention that $\mu_{\ell^h}=\emptyset$ whenever $h\leq 0$.
From the relations above, we deduce that on the $\ell$-volcano having $E, E_1$ and $E_2$ as vertices, $E_1$ and $E_2$ do not lie at the same level. Given the fact that there are at least $\ell-1$ descending rational $\ell$-isogenies parting from $E$ and that $Q$ is any of the $\ell-1$ (or more) $\ell^{n_2}$-torsion points with non-degenerate self-pairing, we conclude that $I_1$ is horizontal or ascending and that $I_2$ is descending.\\
\textit{Case 2.} Suppose now that $k(E)=0$. Note that the case $n_2=1$ was already treated in Proposition~\ref{descendancyProof}. If $n_2>1$, we consider the curve $E$ defined over $\F_{q^{\ell}}$. For $\ell>3$, by Lemma~\ref{TatePairing}b we have $k(E)=1$ for points on $E/\F_{q^{\ell}}$, and we may apply \textit{Case} 1. The case $\ell=2$ is treated inside the proof of Theorem~\ref{classInvariant}.
\end{proof}
\noindent
\begin{remark} If $E$ is a curve lying under the first
stability level and that $E[\ell^{\infty}](\F_q)\simeq
\Z/\ell^{n_1}\Z\times \Z/\ell^{n_2}\Z$, with $n_1>n_2$, then it
suffices to find a point $P_1$ of order $\ell^{n_1}$ and the point
$\ell^{n_1-1}P_1$ generates the kernel of a horizontal or ascending
isogeny ($P_1$ has degenerate self-pairing).
\end{remark}

\noindent
\begin{corollary}
Let $E$ be a curve on an $\ell$-volcano such that the polynomial $\PP_{E,\ell^{n_2}}$ is non-zero over $\F_q$. If $\ell$ is split in the maximal order $\mathcal{O}_{d_K}$, then $E$ is on the crater if and only if $N_{E,\ell^{n_2}}$ is $2$.
Otherwise, $\ell$ is inert in $\mathcal{O}_{d_K}$ if and only if $N_{E,\ell^{n_2}}=0$.
\end{corollary}

\noindent
\textbf{Two stability levels.} Remember that in any irregular volcano,
$v_{\ell}(\#E(\F_q))$ is even and the height $h$ of the volcano is
greater than $v_{\ell}(\#E(\F_q))$. Moreover, all curves at
the top of the volcano have  $E[\ell^{\infty}](\F_q)\simeq
\Z/\ell^{n_2}\Z\times \Z/\ell^{n_2}\Z$ with $n_2=\frac{v_{\ell}(\#E(\F_q))}{2}$.  The existence of a
primitive self-pairing of a $\ell^{n_2}$-torsion point on any curve
lying on the first stability level implies that the polynomial $\PP$
is non-zero at every level from the first stability level up to the
level $\max(h+1-2n_2,0)$ (by Lemma~\ref{ImportantLemma}). We call this level \textit{the second level of stability}. This is illustrated in Figure~\ref{LevelVolcano}. When the second stability level of a volcano is $0$, we say that the volcano is \textit{almost regular}.\\
 Consider now $E$ a curve on the second stability level and
$I:E\rightarrow E_1$ an ascending isogeny. Let $P$ be a
$\ell^{n_2}$-torsion point on $E$ and assume that $T_{\ell^{n_2}}(P,P)\in \mu_{\ell}^*$.
We denote by $\bar{P}\in E(\F_{q^{\ell}})\backslash E(\F_q)$ a point such that $\ell \bar{P}=P$.
By Lemma~\ref{TatePairing} we get
$T_{\ell^{n_2+1}}(\bar{P},\bar{P})$ is a primitive $\ell^{2}$-th
root of unity. It follows by Lemma~\ref{ImportantLemma} that
$T_{\ell^{n_2}}(I(P),I(P))$ is a primitive $\ell$-th root of unity. We
deduce that $\PP_{E_1,\ell^{n_2+1}}$ corresponding to
$E_1/\F_{q^{\ell}}$ is non-zero. Applying this reasoning
repeatedly, we conclude that for every curve $E$ above the second
stability level there is an extension field $\F_{q^{\ell^s}}$ such that
the polynomial $\PP_{E,\ell^{n_2+s}}$ associated to the curve defined
over $\F_{q^{\ell^s}}$ is non-zero. We will show that the degree of this extension field characterizes uniquely curves lying on a fixed level of the volcano, above the second stability level. \\

\vspace{0.3 cm}
\noindent
Let $E$ be an elliptic curve. We suppose that
\begin{eqnarray*}
E(\F_q)[\ell^{\infty}]\simeq \Z/\ell^{n_1}\Z\times \Z/\ell^{n_2}\Z.
\end{eqnarray*}
We define $\mathcal{L}_{\ell,E}$ as follows
$$\mathcal{L}_{\ell,E}=
\left \{\begin{array}{l}
n_1,~\mbox{if}~E~\mbox{is under/on the first stability level} \\
\\
k(E)+1,~\mbox{if}~E~\mbox{is above the first stability level and} \\
~~~~~~~~~~~~~~~~~~\mbox{below the second stability level},\\
\\
-s+1,~\mbox{if $E$ is above the second stability level},
\end{array}
\right.
$$
where $s$ is the smallest integer such that the polynomial $\PP$ of the curve $E$ defined over
$\F_{q^{\ell^s}}$ is nonzero.
\begin{theorem}\label{classInvariant}
Let $E$ be an elliptic curve in \rm{Ell}$_t(\F_q)$. Then $\mathcal{L}_{\ell,E}$ is an invariant of the level of the curve in its $\ell$-volcano.
\end{theorem}
\begin{proof}
\textit{Case 1.} Suppose $\ell \geq 3$.
If $E$ lies below the first stability level, then the structure of the $\ell$-Sylow group of the curve changes from one level to another and $n_1$ characterizes the level of the curve in its $\ell$-volcano.\\
 Suppose now that $E$ lies below the crater, on the first stability level or above it.
Take $P$ and $Q$ two points such that $E[\ell^{n_2}]=\langle P,Q \rangle$ and
 we may assume that $P$ has non-degenerate self-pairing,
 i.e. $T_{\ell^{n_2}}(P,P)\in \mu_{\ell^{k(E)+1}}\backslash \mu_{\ell^{k(E)}}$, and that $Q$ has degenerate self-pairing, i.e. $T_{\ell^{n_2}}(Q,Q)\in \mu_{\ell^{k(E)}}$.
The point $\ell^{n_2-1} Q$ generates the kernel of an ascending isogeny $I: E \rightarrow E'$. We denote by $P'=I(P)$ and, by using Lemma~\ref{ImportantLemma}, we get
\begin{eqnarray*}
T_{\ell^{n_2}}(P',P')\in \mu_{\ell^{k(E)}}\backslash \mu_{\ell^{k(E)-1}}.
\end{eqnarray*}
Note that $P'$ is such that $\ell^{n_2-1}P'$ generates the kernel of $\hat{I}$, which is a descending isogeny. Consequently, the self-pairing of $P'$
is non-degenerate, which means that $k(E')=k(E)-1$.
By Proposition~\ref{existenceNondegenerate}, we have that $k(E)=n_2-1$ if the curve $E$ lies on the first stability level. The reasoning above
implies that $k(E)=n_2-2$ for all curves lying one level above the first stability level.
Iterating this procedure, it also follows that as we ascend from the first stability level to the second one,
the value of $k(E)$ decreases by $1$ at each level. In particular, it equals $0$ at the second stability level and $-1$ at all levels above the second stability level (all self-pairings of curves on these levels are degenerate).\\
Suppose now that $E$ is a curve below the crater, on the second stability level or above it.
We show by induction that if the value of $k(E)$ corresponding to $E$ defined over $\F_{q^{\ell^s}}$ is $0$, then for a curve $E'$ lying one level above the value $k(E')$ is $0$ over $\F_{q^{\ell^{s+1}}}$ and $\F_{q^{\ell^{s+1}}}$ is the smallest extension field with this property. We suppose that $$E(\F_{q^{\ell^s}})[\ell^{\infty}]\simeq \Z/\ell^{n_1}\Z\times \Z/\ell^{n_2}\Z,$$
with $n_1\geq n_2$. We consider $P$ and $Q$ two $\ell^{n_2}$-torsion points such that $\langle P,Q \rangle =E[\ell^{n_2}]$ and that $P$ has primitive self-pairing, while $Q$ has degenerate self-pairing. We denote by $I:E\rightarrow E'$ the ascending isogeny whose kernel is generated by $<\ell^{n_2-1} Q>$ and by $P'=I(P)$. By Lemma~\ref{ImportantLemma} we have
\begin{eqnarray*}
T_{\ell^{n_2}}(P',P')=1.
\end{eqnarray*}
Since $\ell^{n_2-1}P'$ generates the kernel of the dual $\hat{I}$, it follows that $k(E')=-1$ over $\F_{q^{\ell^s}}$.
We denote by $\bar{P}\in E(\F_{q^{\ell^{s+1}}})$ a point such that $\ell \bar{P}=P$.
By Lemma~\ref{TatePairing} we have that
\begin{eqnarray*}
T_{\ell^{n_2+1}}^{(\F_{q^{\ell^{s+1}}})}(\bar{P},\bar{P})\in \mu_{\ell^2}\backslash \mu_{\ell}.
\end{eqnarray*}
By denoting $P''=I(\bar{P})$, we get that
\begin{eqnarray*}
T_{\ell^{n_2+1}}(P'',P'')\in \mu_{\ell}^*.
\end{eqnarray*}
It follows that $k(E')=0$ over $\F_{q^{\ell^{s+1}}}$ and this is the smallest extension field with this property.\\

\noindent
\textit{Case 2.} We treat the case $\ell=2$.
Suppose that $$E[2^{\infty}](\F_q)\simeq \Z/2^{n_1}\Z \times \Z/2^{n_2}\Z.$$
If $n_2>1$, then $$E[2^{\infty}](\F_{q^2})\simeq \Z/2^{n_1+\Delta}\Z \times \Z/2^{n_2+1}\Z.$$
We consider points $\bar{P},\bar{Q}\in E[2^{n_2+1}]$ and $P,Q\in E[2^{n_2}]$ such that $P=2\bar{P}$ and $Q=2\bar{Q}$. Then, by Lemma~\ref{TatePairing}, we have $$T_{2^{n_2+1}}(\tilde{P},\tilde{Q})^2=\pm T_{2^{n_2}}(P,Q).$$
Hence, if $k(E)\geq 2$, the proof is similar to the one of \textit{Case 1}.
We consider the curve $E$ such that $k(E)=1$ and we take a curve $E'$ lying one level above such that there is an ascending isogeny $I:E\rightarrow E'$. Since $k(E)=1$ and the kernel of $I$ is degenerate, then there is a point $P\in E[\ell^{n_2}]$ such that $P'=I(P)$ generates the kernel of $\hat{I}$. By Lemma~\ref{ImportantLemma} we get that
\begin{eqnarray*}
T_{\ell^{n_2}}(P',P')\in \mu_{\ell}^*.
\end{eqnarray*}
Hence the points of the kernel of any descending isogeny starting at $E'$ have self-pairings primitive $\ell$-th roots of unity. Reasoning as in the case $k(E)\geq 2$ over $\F_{q^2}$, we get that $k(E')=1$ over $\F_{q^2}$. A point generating the kernel of an ascending or horizontal isogeny does not have distortion maps (see~\cite[Thm. 2.1]{Charles}). Hence we have
\begin{eqnarray}\label{Bingo}
T_{2^{n_2+1}}^2(\tilde{Q},\tilde{Q})=T_{2^{n_2}}(Q,Q),
\end{eqnarray}
for $Q\in E[\ell^{n_2}](\F_q),\tilde{Q}\in E[\F_{q^2}]$ such that $\ell\tilde{Q}=Q$ and that $\ell^{n_2-1}Q$ generates the kernel of an ascending isogeny. Since $k(E')=1$ for $E'$ defined over $\F_{q^2}$, we get that $T_{\ell^{n_2}}(Q,Q)=1$. We conclude that $k(E')=0$ over $\F_q$. By induction, we may show in a similar manner that there is an extension field over which all curves lying above the second stability level have polynomials $\mathcal{P}$ different from zero.
If $n_2=1$, the first stability level and the second one coincide. If $E$ is a curve on the first stability level of an irregular $2$-volcano (i.e. $q\equiv 1\,\,\textrm{mod}\,\,4$),
we consider the volcano defined over $\F_{q^2}$.
As explained in Remark~\ref{extVolcano2},
 $$E[2^{\infty}](\F_{q^2})\simeq \Z/2^{n_1+\Delta}\Z\times \Z/2^{2}\Z$$ and since the curve lies on the
first stability level, there are points of order $4$ with primitive self-pairing, which implies that for any curve lying one level above, the polynomial $\PP$ is different from zero. Over $\F_{q^2}$, we may reason as in the case $n_2>1$ and show that there is an extension field over which
all curves lying above the first stability level have polynomials $\PP$ different from zero.\\
Finally, if $n_2=1$ and the volcano is regular of height $1$ (i.e. $q\equiv 3~\textrm{mod}~p$), it is obvious that
$\mathcal{L}_{\ell,E}$ is an invariant at every level of the volcano.
\end{proof}

\begin{figure}
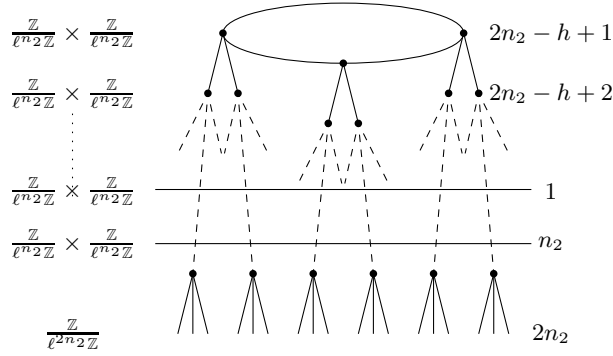

\begin{center}
\begin{pgfpicture}{-4cm}{0cm}{11cm}{6cm}
\pgfsetxvec{\pgfpoint{0.4cm}{0cm}}
\pgfsetyvec{\pgfpoint{0cm}{0.40cm}}
\pgfellipse[stroke]{\pgfxy(5,10)}{\pgfxy(4,0)}{\pgfxy(0,1)}
\pgfputat{\pgfxy(-4,10)}{\pgfbox[center,center]{\small{$\frac{\Z}{\ell^{n_2}\Z}\times \frac{\Z}{\ell^{n_2}\Z} $}}}
\pgfputat{\pgfxy(11.9,10)}{\pgfbox[center,center]{\small{$2n_2-h+1 $}}}
\pgfcircle[fill]{\pgfxy(1,10)}{0.05cm}
\pgfcircle[fill]{\pgfxy(9,10)}{0.05cm}
\pgfcircle[fill]{\pgfxy(5,9)}{0.05cm}
\pgfline{\pgfxy(1,10)}{\pgfxy(0.5,8)}
\pgfline{\pgfxy(1,10)}{\pgfxy(1.5,8)}
\pgfline{\pgfxy(9,10)}{\pgfxy(8.5,8)}
\pgfline{\pgfxy(9,10)}{\pgfxy(9.5,8)}
\pgfline{\pgfxy(5,9)}{\pgfxy(5.5,7)}
\pgfline{\pgfxy(5,9)}{\pgfxy(4.5,7)}
\pgfputat{\pgfxy(-4,8)}{\pgfbox[center,center]{\small{$\frac{\Z}{\ell^{n_2}\Z}\times \frac{\Z}{\ell^{n_2}\Z} $}}}
\pgfputat{\pgfxy(11.9,8)}{\pgfbox[center,center]{\small{$2n_2-h+2$}}}
\pgfcircle[fill]{\pgfxy(0.5,8)}{0.05cm}
\pgfcircle[fill]{\pgfxy(1.5,8)}{0.05cm}
\pgfcircle[fill]{\pgfxy(8.5,8)}{0.05cm}
\pgfcircle[fill]{\pgfxy(9.5,8)}{0.05cm}
\pgfcircle[fill]{\pgfxy(4.5,7)}{0.05cm}
\pgfcircle[fill]{\pgfxy(5.5,7)}{0.05cm}

\pgfxyline(0,2)(0.5,0)
\pgfxyline(0,2)(0,0)
\pgfxyline(0,2)(-0.5,0)
\pgfxyline(2,2)(1.5,0)
\pgfxyline(2,2)(2,0)
\pgfxyline(2,2)(2.5,0)

\pgfxyline(8,2)(8.5,0)
\pgfxyline(8,2)(8,0)
\pgfxyline(8,2)(7.5,0)
\pgfxyline(10,2)(9.5,0)
\pgfxyline(10,2)(10,0)
\pgfxyline(10,2)(10.5,0)
\pgfxyline(11.25,4.8)(-1.25,4.8)
\pgfxyline(11.25,3)(-1.25,3)
\pgfxyline(4,2)(3.5,0)
\pgfxyline(4,2)(4,0)
\pgfxyline(4,2)(4.5,0)
\pgfxyline(6,2)(5.5,0)
\pgfxyline(6,2)(6,0)
\pgfxyline(6,2)(6.5,0)

\pgfsetdash{{0.1cm}{0.1cm}}{0.05cm}
\pgfxyline(0.5,8)(-0.5,6)
\pgfxyline(0.5,8)(0,2)
\pgfxyline(0.5,8)(0.95,6)
\pgfxyline(1.5,8)(1.05,6)
\pgfxyline(1.5,8)(2,2)
\pgfxyline(1.5,8)(2.5,6)

\pgfxyline(8.5,8)(7.5,6)
\pgfxyline(8.5,8)(8,2)
\pgfxyline(8.5,8)(8.95,6)
\pgfxyline(9.5,8)(9.05,6)
\pgfxyline(9.5,8)(10,2)
\pgfxyline(9.5,8)(10.5,6)

\pgfxyline(4.5,7)(3.5,5)
\pgfxyline(4.5,7)(4,2)
\pgfxyline(4.5,7)(4.95,5)
\pgfxyline(5.5,7)(5.05,5)
\pgfxyline(5.5,7)(6,2)
\pgfxyline(5.5,7)(6.5,5)

\pgfsetdash{{0.02cm}{0.1cm}}{0.05cm}
\pgfxyline(-4,7.5)(-4,4.5)

\pgfcircle[fill]{\pgfxy(0,2)}{0.05cm}
\pgfcircle[fill]{\pgfxy(2,2)}{0.05cm}
\pgfcircle[fill]{\pgfxy(4,2)}{0.05cm}
\pgfcircle[fill]{\pgfxy(6,2)}{0.05cm}
\pgfcircle[fill]{\pgfxy(8,2)}{0.05cm}
\pgfcircle[fill]{\pgfxy(10,2)}{0.05cm}
\pgfputat{\pgfxy(-4,4.75)}{\pgfbox[center,center]{\small{$\frac{\Z}{\ell^{n_2}\Z}\times \frac{\Z}{\ell^{n_2}\Z} $}}}
\pgfputat{\pgfxy(11.9,4.75)}{\pgfbox[center,center]{\small{$1$}}}
\pgfputat{\pgfxy(-4,3)}{\pgfbox[center,center]{\small{$\frac{\Z}{\ell^{n_2}\Z}\times \frac{\Z}{\ell^{n_2}\Z} $}}}
\pgfputat{\pgfxy(11.9,3)}{\pgfbox[center,center]{\small{$n_2$}}}
\pgfputat{\pgfxy(-4,0)}{\pgfbox[center,center]{\small{$\frac{\Z}{\ell^{2n_2}\Z} $}}}
\pgfputat{\pgfxy(11.9,0)}{\pgfbox[center,center]{\small{$2n_2$}}}
\end{pgfpicture}
\caption{A level invariant in an $\ell$-volcano}
\label{LevelVolcano}
\end{center}
\end{figure}

We conclude this section by presenting an algorithm which determines the group structure of the $\ell^{\infty}$-torsion group of a curve $E$ (Algorithm~\ref{MethodeVolcanStructure}) and also an algorithm which outputs the
 kernel of a horizontal (ascending) isogeny from $E$,
when
$ \displaystyle E[\ell^{\infty}](\F_q)$ is  given (Algorithm~\ref{MethodeVolcanAscending}).
\begin{algorithm}[h!]
\caption{Computing the structure of the $\ell^{\infty}$-torsion of $E$
  over $\F_q$\hfill \break (assuming volcano height $\geq 1$)}
\label{MethodeVolcanStructure}
\begin{algorithmic}[1]
\REQUIRE A curve $E$ defined over $\F_q$, a prime $\ell$
\item[\textbf{Compute:}] Structure $\Z/\ell^{n_1}\Z\times \Z/\ell^{n_2}\Z$, generators $P_1$ and $P_2$
\STATE Check that $q\equiv 1\pmod{\ell}$ (if not, need to move to
extension field: \textbf{abort})
\STATE Let $t$ be the trace of $E(\F_q)$
\STATE Check $q+1-t\equiv
0\pmod{\ell}$ (if not, consider twist or  \textbf{abort})
\STATE Let $d_{\pi} = t^2-4q,$ let $z$ be the largest integer such
that $\ell^z|d_\pi$ and $h=\lfloor \frac{z}{2}\rfloor$
\STATE Let $n$ be the largest integer such that $\ell^n|q+1-t$ and
$N=\frac{q+1-t}{\ell^n}$
\STATE Take a random point $R_1$ on $E(\F_q)$, let $P_1=N\cdot R_1$
\STATE  Let $n_1$ be the smallest integer such that $\ell^{n_1}P_1=0$
\IF{$n_1=n$}
\STATE  \textbf{Output:} Structure is $\frac{\Z}{\ell^n\Z}$, generator
$P_1.$ \textbf{Exit}
\hfill \break  ($E$ is on the floor, ascending isogeny with kernel $\langle \ell^{n-1}P_1\rangle$)
\ENDIF
\STATE Take a random point $R_2$ on $E(\F_q)$, let $P_2=N\cdot R_2$
and $n_2=n-n_1$
\STATE Let
$\alpha=\log_{\ell^{n_2}P_1}(\ell^{n_2}P_2)\pmod{\ell^{n_1-n_2}}$
\IF {$\alpha$ is undefined}
\STATE \textbf{Goto} 6 ($\ell^{n_2}P_2$ does not belong to $\langle
\ell^{n_2}P_1 \rangle$)
\ENDIF
\STATE Let $P_2=P_2-\alpha P_1$
\STATE  \textbf{If}
$\mbox{WeilPairing}_\ell(\ell^{n_1-1}P_1,\ell^{n_2-1}P_2)=1$
\textbf{goto} 6 (This checks linear independence)
\STATE  \textbf{Output:} Structure is $\frac{\Z}{\ell^{n_1}\Z}\times \frac{\Z}{\ell^{n_2}\Z}$, generators
$(P_1,P_2)$
\end{algorithmic}
\end{algorithm}
\begin{algorithm}[h!]
\caption{Finding the level in the volcano and the kernel of ascending or horizontal isogenies
  \hfill \break (Assuming curve not on floor and below the second
  stability level)}
\label{MethodeVolcanAscending}
\begin{algorithmic}[1]
\REQUIRE A curve $E$, its structure $\frac{\Z}{\ell^{n_1}\Z}\times
\frac{\Z}{\ell^{n_2}\Z}$ and generators $(P_1,P_2)$
\IF {$n_1>n_2$}
\STATE The isogeny with kernel $\langle \ell^{n_1-1}P_1\rangle$ is
ascending or horizontal
\STATE To check whether there is another, continue the algorithm
\ENDIF
\STATE Let $g$ be a primitive $\ell$-th root of unity in $\F_q$
\STATE Let $Q_1=\ell^{n_1-n_2}P_1$
\STATE Let $a=T_{\ell^{n_2}}(Q_1,Q_1),$
$b=T_{\ell^{n_2}}(Q_1,P_2)\cdot T_{\ell^{n_2}}(P_2,Q_1)$ and
$c=T_{\ell^{n_2}}(P_2,P_2)$
\STATE \textbf{If} $(a,b,c)=(1,1,1)$ \textbf{abort} (Above the second stability level)
\STATE Let $\mbox{Count}=0$.
\REPEAT
\STATE Let $a'=a$, $b'=b$ and $c'=c$
\STATE Let $a=a^\ell$, $b=b^\ell$ and $c=c^\ell$
\STATE Let $\mbox{Count}=\mbox{Count}+1$
\UNTIL {$a=1$ and $b=1$ and $c=1$}
\STATE Let $L_a=\log_g(a')$, $L_b=\log_g(b')$ and $L_c=\log_g(c')\pmod{\ell}$
\STATE Let $\PP(x,y)=L_ax^2+L_bxy+L_cy^2 \pmod{\ell}$
\STATE \textbf{If} $n_1=n_2$, let $\mbox{\it
  LevelInvariant}=\mbox{Count}-1$ \textbf{else} $\mbox{\it LevelInvariant}=n_1$
\STATE \textbf{Output:} Level Invariant ($\mathcal{L}_{\ell,E}$) is {\it LevelInvariant}
\STATE \textbf{If} $\PP$ has no homogenous roots modulo $\ell$, \textbf{Output:}
No isogeny (a single point on the crater)
\STATE \textbf{If} single root $(x_1,x_2)$ \textbf{Output:}
One isogeny with kernel $\langle \ell^{n_2-1}(x_1Q_1+x_2P_2)\rangle$
\IF {$\PP$ has two roots  $(x_1,x_2)$ and  $(y_1,y_2)$}
\STATE Two isogenies with kernel $\langle
\ell^{n_2-1}(x_1Q_1+x_2P_2)\rangle$ and $\langle \ell^{n_2-1}(y_1Q_1+y_2P_2)\rangle$
\ENDIF
\end{algorithmic}
\end{algorithm}

We assume that the height of the volcano is $h\leq 2n_2+1$, or, equivalently, that
the curve $E$ lies on or below the second stability level, which implies that the polynomial $\PP$ is non-zero at every level in the volcano. This allows us to distinguish between different directions of $\ell$-isogenies departing from $E$. Algorithm~\ref{MethodeVolcanAscending} computes the level in the volcano of the curve $E$, which is equivalent to computing the level invariant~$\mathcal{L}_{\ell,E}$. \\
Of course, similar algorithms can be given for curves lying above the second stability level, but in this case we need to consider the volcano over an extension field $\F_{q^{\ell^s}}$. Since computing points defined over extension fields of degree greater than $\ell$ is expensive, our complexity analysis in Section~\ref{WalkingOnTheVolcano} will show that it is more efficient to use Kohel's and Fouquet-Morain algorithms to explore the volcano until the second level of stability is reached and to use Algorithms~\ref{MethodeVolcanStructure} and~\ref{MethodeVolcanAscending} afterwards.

\section{Walking the volcano: modified algorithms}\label{WalkingOnTheVolcano}
As mentioned in the introduction, several applications of isogeny
volcanoes have recently been proposed. These applications require the
ability to walk descending and ascending paths on the volcano and also
to walk on the crater of the volcano. We recall that a \textit{path} is a
sequence of isogenies that never backtracks.
We start this section with a brief description of existing algorithms
for these tasks, based on methods given by Kohel~\cite{Kohel} and by
Fouquet and Morain in~\cite{FouMor}. We present modified
algorithms, which rely on the method presented in Algorithm~\ref{MethodeVolcanAscending} to find ascending or horizontal
isogenies and to compute the level invariant $\mathcal{L}_{\ell,E}$. Then, we give complexity analysis for these algorithms and show that in many cases our method is competitive. Finally, we give two concrete examples in which the new algorithms can walk the crater of an isogeny volcano very efficiently compared to existing algorithms.

\subsubsection*{A brief description of existing algorithms.} Existing
algorithms rely on three essential properties in isogeny
volcanoes. Firstly, it is easy to detect that a curve lies on the floor
of a volcano, since in that case, there is a single isogeny from this
curve. Moreover, this isogeny can only be ascending (or horizontal if
the height is~$0$). Secondly, if in an arbitrary path in a volcano there is a
descending isogeny, then all the subsequent isogenies in the path are
also descending. Thirdly, from a given curve, there is either exactly one ascending
isogeny or at most two horizontal ones. As a consequence, finding a
descending isogeny from any curve is easy: it suffices to walk three
paths in parallel until one path reaches the floor. This shortest path
is necessarily descending and its length gives the level of the starting
curve in the volcano. To find an ascending or horizontal isogeny,
the classical algorithms try all possible isogenies until they find
one which leads to a curve either at the same level or above the
starting curve. This property is tested by constructing descending paths from
 all the neighbours of the initial curve and picking the curve which gave the longest path.

Note that alternatively, one could walk in parallel all of the
$\ell+1$ paths
starting from the initial curve and keep the (two) longest as
horizontal or ascending. As far as we know, this has not been proposed
in the literature, but this variant of existing algorithms offers a
slightly better asymptotic time complexity. For completeness, we give
a pseudo-code description of this parallel variant of Kohel and
Fouquet-Morain algorithms as Algorithm~\ref{AlgoParallel}.

\begin{algorithm}[h!]
\caption{Parallel variant of ascending/horizontal step \hfill \break (using modular polynomials)}
\label{AlgoParallel}
\begin{algorithmic}[1]
\REQUIRE A $j$-invariant $j_0$ in $\F_q$, a prime $\ell$, the modular
polynomial $\Phi_\ell(X,Y)$.
\STATE Let $f(x)=\Phi_\ell(X,j_0)$
\STATE Compute $J_0$ the list of roots of $f(x)$ in $\F_q$
\STATE \textbf{If} $\#J_0=0$ \textbf{Output:} ``Trivial volcano'' \textbf{Exit}
\STATE \textbf{If} $\#J_0=1$ \textbf{Output:} ``On the floor, step leads to:'', $J_0[1]$ \textbf{Exit}
\STATE \textbf{If} $\#J_0=2$ \textbf{Output:} ``On the floor, two
horizontal steps to:'', $J_0[1]$ and $J_0[2]$ \textbf{Exit}
\STATE Let $J=J_0$. Let $J'$ and $K$ be empty lists. Let  $\mbox{Done}=\textbf{false}$.
\REPEAT
\STATE Perform multipoint evaluation of $\Phi_\ell(X,j)$, for each
$j\in J$. Store in list $F$
\FOR{$i$ from $1$ to $\ell+1$}
\STATE Perform partial factorization of $F[i]$, computing at most two roots
$r_1$ and $r_2$
\IF{$F[i]$ has less than two roots}
\STATE  Let $\mbox{Done}=\textbf{true}$. Append $\bot$ to $K$ (Reaching floor)
\ELSE
\STATE \textbf{If} $r_1\in J'$ \textbf{then} append $r_1$ to $K$ \textbf{else} append $r_2$ to $K$. (Don't backtrack)
\ENDIF
\ENDFOR
\STATE Let $J'=J$, $J=K$ and $K$ be the empty list
\UNTIL{Done}
\STATE \textbf{for each} $i$ from $1$ to $\ell+1$ such that $J[i]\neq
\bot$ append $J_0[i]$ to $K$
\STATE \textbf{Output:} ``Possible step(s) lead to:'' $K$ (One or two outputs)
\end{algorithmic}
\end{algorithm}

\subsubsection*{Basic idea of the modified algorithms.}
In our algorithms, we first need to choose a large enough extension
field to guarantee that the kernels of all required isogenies are
spanned by $\ell$-torsion points defined on this extension field.  As
explained in Corollary~\ref{existencelTorsion}, the degree $r$ of this
extension field is the order of $q$ modulo $\ell$ and it can be computed
very quickly after factoring $\ell-1$. As usual, we choose an
arbitrary irreducible polynomial of degree $r$ to represent $\F_{q^r}$.
 Points of order $\ell$ are computed by running Algorithm~\ref{MethodeVolcanStructure}, this time over $\F_{q^r}$. Once
this is done, assuming that we are starting from a curve below the
second level of stability, we use Algorithm~\ref{MethodeVolcanAscending} to find all ascending or horizontal isogenies from the initial curve.
 In order to walk a descending path, it suffices to choose any other isogeny. Note that, in the subsequent steps of a descending path, in the cases where the group structure satisfies $n_1>n_2$, it is not necessary to run
Algorithm~\ref{MethodeVolcanAscending} as a whole. Indeed, since we
know that we are not on the crater, there is a single ascending
isogeny and it is spanned by $\ell^{n_1-1}P_1$.
Note that in order to determine the level of the curve in the volcano and hence the $\ell$-adic valuation of the endomorphism ring we do not need to take any steps on the volcano. Indeed, Algorithm~\ref{MethodeVolcanAscending} computes the level invariant $\mathcal{L}_{\ell,E}$ with three pairing computations and several exponentiations to the power $\ell$. Finally, above the second stability level, we have two options. In theory, we can consider curves over larger extension fields (in order to get polynomials $\PP\neq 0$). Note that this is too costly in practice. Therefore, we use preexisting algorithms, but it is not necessary to follow descending paths all the way to the floor. Instead, we can stop
these paths at the second stability level, where our methods can be used.

\subsubsection*{Computing endomorphism rings}
Kohel~\cite{Kohel} describes a deterministic algorithm to compute the endomorphism ring of an elliptic curve. For small values of $\ell$ and when a large power of $\ell$ divides the conductor of the endomorphism ring, he uses algorithms traveling on isogeny volcanoes to find the shortest path from the curve to the floor and thus determine the level of the curve in the volcano. We propose replacing the descent to the floor by a computation of the level invariant $\mathcal{L}_{\ell,E}$. On an almost regular volcano this is done by computing the structure of the $\ell$-Sylow group and then by computing the value of $k(E)$.

\subsection{Complexity analysis}
\subsubsection*{Computing a single isogeny.}
Before analyzing the complete algorithms, we first compare the costs
of taking a single step on a volcano by using the two methods existing
in the literature: modular polynomials and classical V\'elu's formulae.
Suppose that we wish to take a step from a curve $E$. With the modular polynomial approach, we have to evaluate the polynomial $f(X)=\Phi_{\ell}(X,j(E))$ and find its roots in $\F_q$. Assuming that the modular polynomial (modulo the characteristic of $\F_q$) is given as input and using asymptotically fast probabilistic algorithms to factor $f(X)$, the
cost of a step in terms of arithmetic operations in $\F_q$ is
$O(\ell^2+M(\ell)\log{q}),$ where $M(\ell)$ denotes the operation
count of multiplying polynomials of degree $\ell$. In this formula,
the first term corresponds to evaluation of
$\Phi_{\ell}(X,j(E_{i-1}))$ and the second term to root
finding\footnote{Completely splitting $f(X)$ to find all its roots
would cost $O(M(\ell)\log{\ell}\log{q})$, but this is reduced to
$O(M(\ell)\log{q})$ because we only need a constant number of roots
for each polynomial $f(X)$.}.

With V\'elu's formulae, we need to take into account the fact that the
required $\ell$-torsion points are not necessarily defined over
$\F_q$. Let $r$ denotes the smallest integer such that the required
points are all defined over $\F_{q^r}$. We know that $1\leq r\leq
\ell-1$. Using asymptotically efficient algorithms to perform
arithmetic operations in $\F_{q^r}$, multiplications in $\F_{q^r}$ cost
$M(r)$ $\F_q$-operations. Given an $\ell$-torsion point $P$ in
$E(\F_{q^r})$, the cost of using V\'elu's formulae is $O(\ell)$
operations in $\F_{q^r}$. As a consequence, in terms of $\F_q$
operations, each isogeny costs $O(\ell M(r))$ operations.
As a consequence, when $q$ is not too large and $r$ is close to
$\ell$, using V\'elu formulae is more expensive by a logarithmic
factor.

\subsubsection*{Computing an ascending or horizontal path.}
With the classical algorithms, each step in an ascending or horizontal
path requires considering all the $O(\ell)$ neighbours of the curve
and testing each of them by walking descending paths of height bounded by $h$. The expected cost of each descending
path is $O(h(\ell^2+M(\ell)\log{q}))$ and the total cost is
$O(h(\ell^3+\ell M(\ell)\log{q}))$
(see~\cite{Kohel,Sutherland1}). When $\ell >> \log{q}$, this cost is dominated by
the evaluations of the polynomial $\Phi_\ell$ at each
$j$-invariant. Thus, by walking in parallel $\ell+1$ paths from the
original curve, we can amortize the evaluation of $\Phi_{\ell}(X,j)$
over many $j$-invariants using fast multipoint evaluation,
see~\cite[Section 3.7]{Montgomery} or~\cite{ShoupvzG}, thus replacing
$\ell^3$ by $\ell\,M(\ell)\log{\ell}$ and reducing the complexity of a
step to $O(h \ell\,M(\ell)(\log{\ell}+\log{q}))$. However, this
increases the memory requirements.

With our modified algorithms, we need to find the $\ell^{\infty}$-structure of each curve, compute some discrete logarithms in $\ell$-groups, perform a
small number of pairing computations and compute the roots of
$\PP_{E,\ell^{n_2}}$.  Except for the computation of discrete
logarithms, it is clear that all these additional operations are
polynomial in $n_2$ and $\log{\ell}$ and they take negligible time
in practice (see Section~\ref{section:example}). Using generic algorithms, the
discrete logarithms cost $O(\sqrt{\ell})$ operations, and this can be reduced to $\log{\ell}$  by storing a sorted table of
precomputed logarithms. After this is done, we have to compute at most
two isogenies, ignoring the one that backtracks. Thus, the computation
of one ascending or horizontal step is dominated by the computation of
isogenies and costs $O(\ell M(r))$.

For completeness, we also mention the complexity analysis of
Algorithm~\ref{MethodeVolcanStructure}. The dominating step here is
the multiplication by $N$ of randomly chosen points. When we
consider the curve over an extension field $\F_{q^r}$, the expected cost is
$O(r\log{q})$ operations in  $\F_{q^r}$, i.e. $O(rM(r)\log{q})$
operations in $\F_q$.

Finally, comparing the two approaches on a regular volcano, we see that even in
the less favorable case, we gain a factor $h$ compared to the
classical algorithms. More precisely,  the two
are comparable, when the height $h$ is
small  and $r$ is close to $\ell$. In all the other cases, our modified algorithms are more efficient. This analysis is summarized in
Table~\ref{comparaisonGlobalEvaluation}. For compactness $O(\cdot)$s are
omitted from the table.
\begin{table}
\caption{\label{comparaisonGlobalEvaluation} Walking the volcano:
  Order of the expected cost per step}
\centering
\begin{tabular}{|c|c|c|c|}
\hline \hline
\,\, & \multicolumn{2}{c|} {Descending path} & Ascending/Horizontal\\
\hline
\,\, & One step & Many steps &\\
\hline
\cite{Kohel,FouMor} & $h(\ell^2+M(\ell)\log q)$ & $(\ell^2+M(\ell)\log q)$ &
$h(\ell^3+\ell\,M(\ell)\log q)$\\
Parallel evaluation & -- & -- & $h \ell\,M(\ell)(\log{\ell}+\log{q})$ \\
\hline
Regular volcanoes& \multicolumn{3}{c|}{\mbox{Structure determination}}
\\
\cline{2-4}
Best case & \multicolumn{2}{c|} {$\log{q}$}& $\log{q}$\\
Worst case $r\approx \ell/2$&\multicolumn{2}{c|} { $r\,M(r) \log{q}$ }&
 $r\,M(r) \log{q}$\\
\hline
Regular volcanoes& \multicolumn{3}{c|}{\mbox{Isogeny construction}} \\
\cline{2-4}
Best case &\multicolumn{2}{c|} {$\ell$}& $\ell$\\
Worst case $r\approx \ell/2$& \multicolumn{2}{c|} {$r\,M(r)$ }&
 $r\,M(r)$\\
\hline
 Irregular volcanoes & \multicolumn{3}{c|}{\mbox{}}\\
(worst case) & \multicolumn{3}{c|}{\mbox{No
     improvement}} \\
\hline
\end{tabular}
\end{table}

\subsubsection*{Computing endomorphism rings}
On a regular volcano, computing the invariant $\mathcal{L}_{\ell,E}$ involves computing the group structure and some pairings. Hence, the expected running time of the computation is $O(rM(r)\log q+n_2\log \ell)$, while the complexity of Kohel's algorithm is $O(h(\ell^2+M(\ell)\log q))$.


\subsubsection*{Irregular volcanoes.}
Consider a fixed value of $q$ and let $s=v_{\ell}(q-1)$. First of all,
note that all curves lying on irregular volcanoes satisfy
$\ell^{2s}|q+1-t$ and $\ell^{2s+2}|t^2-4q$. For traces that satisfy
only the first condition, we obtain a regular volcano. We estimate the
total number of different traces of elliptic curves lying on
$\ell$-volcanoes by
$\#\{t~\mbox{s.t.}~\ell^{2s}|q+1-t\,\,\textrm{and}\,\,t\in
[-2\sqrt{q},2\sqrt{q}]\}\sim \frac{4\sqrt{q}}{\ell^{2s}}.$\smallskip

\noindent Next, we
estimate traces of curves lying on irregular volcanoes
by
$$~~~~~~~~~~~\#\{t~\mbox{s.t.} ~\ell^{2s}|q+1-t~,\ell^{2s+2}|t^2-4q\,\,\textrm{and}\,\,t\in
[-2\sqrt{q},2\sqrt{q}]\}\sim \frac{4{\sqrt{q}}}{\ell^{2s+2}}.$$
Indeed, by writing $q=1+\gamma \ell^s$ and $t=2+\gamma \ell^s+\mu
\ell^{2s}$, and imposing the condition $\ell^{2s+2}|t^2-4q$, we find
that $t\cong t_0(\gamma,\mu)( \textrm{mod}\,\,\ell^{2s+2})$.

Thus, we estimate the probability of picking a curve whose volcano is not regular, among curves lying on volcanoes of height greater than $0$, by $\frac{1}{\ell^2}$. (This is a crude  estimate because the
number of curves for each trace is proportional to the
Hurwitz class number\footnote{See~\cite[Th.  14.18]{Cox} for $q$ prime.} $H(t^2-4q)$). This probability is not negligible for small values of $\ell$. However, since our method also works everywhere on almost regular volcano, the
probability of finding a volcano where we need to combine our modified
algorithm with the classical algorithms is even lower.
Furthermore, in some applications, it is possible to restrict
ourselves to regular volcanoes.

\subsection{Some practical examples}
\label{section:example}
\subsubsection*{A favorable case.}
In order to demonstrate the potential of the modified algorithm, we
consider the favorable case of a volcano of height $2$, where all the
necessary $\ell$-torsion points are defined over the base field $\F_p$,
where $p=619074283342666852501391$ is prime. We choose $\ell=100003$.\smallskip

\noindent Let $E$ be the elliptic curve whose Weierstrass equation is
$$~~~~y^2=x^3+198950713578094615678321\,x+32044133215969807107747.$$
The group $E[\ell^\infty]$ over $\F_p$ has structure $\frac{\Z}{\ell^4
  Z}.$ It is spanned by the point
$$P=(110646719734315214798587 , 521505339992224627932173).$$
Taking the $\ell$-isogeny $I_1$ with kernel $\langle \ell^3 P\rangle$,
we obtain the curve
$$E_1: y^2 = x^3 + 476298723694969288644436\,x + 260540808216901292162091,
$$
with structure of the $\ell^\infty$-torsion $\frac{\Z}{\ell^3\Z}\times
\frac{\Z}{\ell\Z}$ and generators
$$P_1 = (22630045752997075604069 , 207694187789705800930332)~\mbox{and}$$
$$Q_1 = (304782745358080727058129 ,193904829837168032791973).$$

\noindent The $\ell$-isogeny $I_2$ with kernel $\langle \ell^2 P_1\rangle$ leads
to the curve
$$
E_2: y^2 = x^3 + 21207599576300038652790\,x + 471086215466928725193841,
$$
on the volcano's crater and with structure $\frac{\Z}{\ell^2\Z}\times
\frac{\Z}{\ell^2\Z}$ and generators
$$P_2  =  (545333002760803067576755 ,367548280448276783133614 )~\mbox{and}$$
$$Q_2= (401515368371004856400951 , 225420044066280025495795 ).$$
\smallskip

\noindent Using pairings on these points, we construct the polynomial:
$$\PP(x,y)=97540\,x^2 + 68114\,x\,y + 38120\, y^2,
$$
having homogeneous roots $(x,y)=(26568, 1)$ and $(72407, 1)$. As a
consequence, we have two horizontal isogenies with kernels
$\langle \ell(26568\,P_2+Q_2)\rangle $ and $\langle \ell(72407\,P_2+Q_2)\rangle$.
We can continue and make a complete walk around the crater which
contains $22$ different curves. Using a simple implementation under
Magma~2.15-15, a typical execution takes about 134 seconds\footnote{This timing varies
  between executions. The reason that we first try one root of $\PP$,
  if it backtracks on the crater, we need to try the other one. On
  average, $1.5$ root is tried for each step, but this varies
  depending on the random choices.} on a single core of an Intel Core 2 Duo
at $2.66$~GHz. Most of the time is taken by the computation of
V\'elu's formulas (132 seconds) and the computation of discrete
logarithms (1.5 seconds) which are not tabulated in the
implementation. The computation of pairings only takes 20 milliseconds.

\subsubsection*{A larger example.} We have also implemented the computation
for $\ell=1009$ using an elliptic curve with $j$-invariant
$j=34098711889917$ in the prime field defined by
$p=953202937996763$. The $\ell$-torsion appears in a extension field
of degree $84$. The $\ell$-volcano has height two and the crater
contains 19 curves. Our implementation walks the crater in 20
minutes. More precisely, 750 seconds are needed to generate the
curves' structures, 450 to compute V\'elu's formulas, 28 seconds for
the pairings and 2 seconds for the discrete logarithms.

\subsubsection*{Computing the endomorphism ring}

Our benchmarks show that our method is very efficient in the favorable case,
i.e. when the $\ell$-torsion points are defined over the base field. Otherwise, if $\ell$
is small, the efficiency of our method depends asymptotically on the ratio $h/r$. We have implemented
our algorithm and Kohel's method with MAGMA and ran experiments for various values of $h,r$ and $\ell$.
Results are given in Table~\ref{timings}. For large $\ell$ ($\ell \geq 2^{10}$), we could not test Kohel's
method since modular polynomials may not be precomputed with MAGMA.
\begin{table}
\caption{\label{timings} Endomorphism ring computation: Benchmarks}
\centering
\begin{tabular}{|c|c|c|c|}
\hline
\hline
Parameters & Kohel & This work \\
\hline
$D=1009$, $\ell=31$, $h=10$, $r=1$ &  1.80 s  &  0.01 s    \\
\hline
$D=1009$, $\ell=101$, $h=3$, $r=10$ & 1.18 s  &  0.75 s  \\
\hline
$D=1009$, $\ell=31$,  $h=6$, $r=5$ &  1.15 s    & 0.33 s   \\
\hline
$D=4*919$, $h=2$, $\ell=1009$, $r=84$ &   -  & 43 s \\
\hline
\end{tabular}
\end{table}

\subsubsection*{An example}
For curves such that the index of $\Z[\pi]$ is divisible by a large power
of a small prime $\ell$, we use Kohel's algorithm combined with our method to
compute the largest power of $\ell$ dividing the conductor of the endomorphism ring.
Suppose we are given the curve with j-invariant $$j_0=71892495629450480796525055574120577929291359932$$
over the prime field defined by $$p=555574087029024034910907703752286309950415657009.$$
The discriminant of $\Z[\pi]$ is
$$d_{\pi}=2^2\cdot 31^{30}\cdot 1009,$$
hence the height of the $31$-volcano is 15.
The $31$-Sylow group structure is $\frac{\Z}{31^3\Z}\times \frac{\Z}{31^3 \Z}$ and the corresponding $k(E)=-1$,
hence we may not determine the level of the curve in the $31$-volcano by using pairings over $\F_p$. We could move to $\F_{p^{\ell}}$ and compute pairings over this field, but it is rather expensive. Instead, we use Kohel's algorithm to find the shortest path to the second stability level. For each curve we consider, we compute the corresponding pairings over $\F_p$ to see whether we get a polynomial $\PP$ different from zero. When we get such a polynomial, we stop because we have reached the second stability level. For example, a random walk in the volcano produces a shortest path to the second stability level given by the curves with $j$-invariants
$$j_1=304777814376748778212312171834280090074154445427~\mbox{and}~k(E_1)=-1,$$
$$j_2=191449283692968031770360270038328919070842850348~\mbox{and}~k(E_2)=-1,$$
$$j_3=500824144736236330809586376475032618300606767898~\mbox{and}~k(E_3)=-1,$$
$$j_4=439660047668527271074847223836176503148636315832~\mbox{and}~k(E_4)=0.$$
The curve $E_4$ lies on the second stability level, hence we deduce that the $31$-valuation of the index of $\Z[\pi]$ in
$\textrm{End}(E)$ is $9$.

\section{Conclusion and perspectives}

In this paper, we have proposed a method which allows one to determine, given a curve $E$ in the regular part of an isogeny volcano and an $\ell$-torsion point $P$ on the curve, the type of the $\ell$-isogeny whose kernel is spanned by $P$. In addition, this method permits one to find the ascending isogeny (or horizontal isogenies) from $E$, if a basis for the $\ell$-torsion is given. We expect that this method can be used to improve the performance of several volcano-based algorithms, such as the computation of the Hilbert class polynomial~\cite{Sutherland1} or of modular polynomials~\cite{Sutherland3}.\\
Finally, on an $\ell$-volcano, we have given a level invariant which can be determined by computing the structure of the $\ell$-Sylow group and
 a small number of pairings. This gives a new method to compute the $\ell$-adic valuation of the conductor of the endomorphism ring of an elliptic
  curve, for small values of $\ell$, and may thus be used in algorithms computing the endomorphism ring of an elliptic curve.\\

\noindent
\textbf{Acknowledgments.}
The authors thank Jean-Marc Couveignes for the idea in the proof of Lemma 1 and anonymous reviewers of the conference version~\cite{IonJou} for helpful comments. The first author is grateful to Ariane M\'ezard for many discussions on number theory and isogeny volcanoes, prior to this work.
 This work was partially supported by the French Agence Nationale de la Recherche through the ECLIPSES project under Contract ANR-09-VERS-018 and by the Direction G\'en\'erale de l'Armement through the AMIGA project under Contract 2010.60.055.

\bibliographystyle{plain}
\bibliography{sorina}

\end{document}